\documentclass[10pt]{amsart}
\usepackage{amsthm, amsmath,cleveref, amsfonts, amssymb, enumerate, textcmds, tensor, enumitem, mathdots}
\usepackage{pst-node}
\usepackage{tikz-cd} 
\usepackage{graphicx,tikz}
\usepackage{enumerate}
\usepackage{pinlabel}
\usepackage[colorinlistoftodos]{todonotes}
\newcommand{\Z}{\mathbb Z}
\newcommand{\R}{\mathbb R}
\newcommand{\C}{\mathbb C}

\newcommand{\p}{\mathfrak{p}}
\newcommand{\g}{\mathfrak{g}}
\newcommand{\kk}{\mathfrak{k}}
\newcommand{\aaa}{\mathfrak{a}}
\newcommand{\h}{\mathfrak{h}}
\newcommand{\Met}{\mathbb{P}\mathrm{Met}}

\newtheorem{thm}{Theorem}[section]
\newtheorem{lem}[thm]{Lemma}
\newtheorem{prop}[thm]{Proposition}
\newtheorem{cor}[thm]{Corollary}

\newtheorem*{thma}{Theorem A}
\newtheorem*{LabourieConjecture}{Labourie Conjecture}
\newtheorem*{thmb}{Theorem B}
\newtheorem*{thmbp}{Theorem B for $\textrm{PGL}(n,\C)$}

\newtheorem*{cora}{Corollary A}

\DeclareMathOperator{\tr}{\textrm{tr}}

\DeclareMathOperator{\dbar}{\overline{\partial}}
\DeclareMathOperator{\pK}{\mathcal{O}(\p^{\C})^{K^{\C}}}
\DeclareMathOperator{\aW}{\mathcal{O}(\aaa^{\C})^W}

\theoremstyle{definition}
\newtheorem{defn}[thm]{Definition}
\newtheorem{remark}[thm]{Remark}

\begin{document}
\title{Unstable minimal surfaces in symmetric spaces of non-compact type}
\author{Nathaniel Sagman}
\author{Peter Smillie}

\begin{abstract}
We prove that if $\Sigma$ is a closed surface of genus at least 3 and $G$ is a split real semisimple Lie group of rank at least $3$ acting faithfully by isometries on a symmetric space $N$, then there exists a Hitchin representation $\rho:\pi_1(\Sigma)\to G$ and a $\rho$-equivariant unstable minimal map from the universal cover of $\Sigma$ to $N$. This follows from a new lower bound on the index of high energy minimal maps into an arbitrary symmetric space of non-compact type. Taking $G=\textrm{PSL}(n,\R)$, $n\geq 4$, this disproves the Labourie Conjecture.
\end{abstract}
\maketitle

\begin{section}{Introduction}

Higher Teichm\"uller theory is centered on the study of representations of the fundamental group of a closed surface $\Sigma_g$ of genus $g$, $g\geq 2,$ into a real semisimple Lie group $G$ of rank at least 2. More specifically, a higher Teichm\"uller space is a connected component of the representation variety $\mathrm{Rep}(\Sigma_g,G) = \mathrm{Hom}(\pi_1(\Sigma_g),G)//G$ consisting entirely of discrete and faithful representations \cite{Wie}. The first examples, the Hitchin components, consisting of the Hitchin representations, were introduced in 1992 in the seminal paper \cite{Hi} using the theory of Higgs bundles. It was a decade later that Labourie \cite[Theorem 1.5]{L0} and Fock-Goncharov \cite{FG} independently proved that Hitchin components are higher Teichm\"uller spaces in the sense above. Other examples come from maximal and $\Theta$-positive representations (see \cite{BCGGO} and \cite{GLW}). Each of these families generalize the classical Teichm\"uller space $\mathbf{T}_g \subset \mathrm{Rep}(\Sigma_g,\mathrm{PSL}(2,\R))$ of Fuchsian representations. For an introduction to higher Teichm{\"u}ller theory, see the survey \cite{Wie}.

One of the fundamental goals of higher Teichm\"uller theory is to understand how the many interconnected properties - dynamical, geometric, and complex analytic - of the classical Teichm\"uller space generalize to higher Teichm\"uller spaces. For example, in \cite{Hi} Hitchin inquires about both the geometric significance of Hitchin representations and the nature of the natural action of the mapping class group on the Hitchin component. To this end, a central focus of the field for the last twenty years has been the Labourie Conjecture. This concerns minimal surfaces, meaning surfaces that locally minimize area. It predicts that for each Hitchin representation $\rho$ into $\mathrm{PSL}(n,\R)$, there should be a unique $\rho$-invariant minimal surface in the symmetric space whose quotient by $\rho$ is $\Sigma_g$. Labourie explained that when the conjecture holds for a given genus $g$ and dimension $n$, then this Hitchin component admits a mapping class group equivariant parametrization as the total space of a certain natural holomorphic vector bundle over Teichm{\"u}ller space. Moreover, this bundle comes with a mapping class group invariant K{\"a}hler metric for which Teichm{\"u}ller space is totally geodesic and whose restriction to Teichm{\"u}ller space is the Weil–Petersson metric. 
Labourie soon proved \emph{existence} of such a $\rho$-invariant minimal surface for any Hitchin representation $\rho$, leaving open only \emph{uniqueness}. Shortly after, he proved uniqueness for $n=3$ (see also \cite{Lo}).

Labourie's existence theorem for minimal surfaces in fact applies to a broad class of so-called Anosov representations that includes all representations in known higher Teichm{\"u}ller spaces. It is natural to extend the uniqueness conjecture to all higher Teichm{\"u}ller spaces, and we refer to this extension as the Generalized Labourie Conjecture. This is now known to be true for all known higher Teich{\"u}ller spaces for all Lie groups $G$ of rank 2 \cite{CTT}. For more background on the Labourie Conjecture, see section \ref{sec: labconjecture} below or the surveys \cite[section 6]{Wie}, \cite[section 6.5]{Swoboda}, \cite[section 7]{Li}.

As far as we're aware, the Generalized Labourie Conjecture was broadly believed to be true until March 2021, when Markovi{\'c} demonstrated the existence of a product of Fuchsian representations into $\textrm{PSL}(2,\mathbb{R})^3$ with multiple minimal surfaces in the corresponding product of closed hyperbolic surfaces \cite{M2}, for $g$ sufficiently large (see also \cite{M1}). This shows that the generalized conjecture fails for a certain higher Teichm\"uller space in rank 3, namely $\mathbf{T}_g \times \mathbf{T}_g \times \mathbf{T}_g$. With Markovi{\'c} in \cite{MSS}, we proved a suite of results elaborating on and improving on \cite{M2}.

In this paper, we show that Labourie's original conjecture fails for all $n\geq 4$, and that the generalized version is not true for any Hitchin component of any real Lie group of rank at least 3 (in genus at least 3, see Remark \ref{rem: genus 2}). In other words, the uniqueness conjecture fails for some representation in each Hitchin component except for those components in which it was already known to hold. 

Our basic strategy is to use Higgs bundles to construct unstable minimal surfaces invariant by some representation $\rho$ in a given Hitchin component. A minimal surface is called unstable if the area can be decreased to second order along a one-parameter variation of ($\rho$-invariant) surfaces. Labourie's existence theorem for minimal surfaces actually proves the existence of a surface that is area-minimizing among all competitors, and therefore cannot be unstable. Hence the uniqueness conjecture fails for any Hitchin representation admitting an unstable invariant minimal surface.  In fact, we provide in Theorem B a general way to produce unstable minimal surfaces in symmetric spaces of non-compact type, which is not specialized to Hitchin representations or even discrete representations. Our method also shows that uniqueness fails for many other higher Teichm\"uller spaces, and we expect that the Generalized Labourie Conjecture should also fail except in those cases in which it is known to hold.

\subsection{The Labourie Conjecture}\label{sec: labconjecture}
We now give a more detailed overview of the theory around the Labourie Conjecture. 
Let $\Sigma_g$ denote a closed surface of genus $g\geq 2$ and let $\mathbf{T}_g$ be the Teichm{\"u}ller space of marked complex structures on $\Sigma_g$. Let $G$ be a real semisimple Lie group with rank at least $3$ acting faithfully by isometries on a symmetric space $N$ of non-compact type, and consider a representation $\rho: \pi_1(\Sigma_g) \to G$. For every Riemann surface structure $S$ on $\Sigma_g$, with universal cover $\tilde{S}$, and $\rho$-equivariant smooth map $f:\tilde{S}\to N,$ there is a well-defined notion of $\rho$-equivariant Dirichlet energy $\mathcal{E}(S,f)$ (see section 2.2). Donaldson \cite{D} and Corlette \cite{C} proved that for essentially every $\rho$, there is a unique $\rho$-equivariant harmonic map $h: \tilde{S} \to N$, which satisfies $$\mathcal{E}(S,h)=\inf_f \mathcal{E}(S,f).$$ This gives a function $\mathbf{E}_\rho: \mathbf{T}_g \to [0,\infty)$, by $\mathbf{E}_\rho(S)=\mathcal{E}(S,h)$. When $S$ is a non-zero critical point of $\mathbf{E}_\rho$, the harmonic map is (weakly) conformal and as a result its image is a (branched) minimal surface. Conversely every branched minimal surface is a weakly conformal harmonic map for the Riemann surface structure induced by the (possibly degenerate) pullback metric. Moreover, since the energy of a conformal map is equal to the area of its image, a point in $\mathbf{T}_g$ minimizes $\mathbf{E}_\rho$ if and only if the branched minimal surface is area-minimizing. Henceforth we define a minimal map to be one that is harmonic and weakly conformal.

The Labourie Conjecture concerns Hitchin representations into $\textrm{PSL}(n,\R)$, which we define now. Let $i_n: \mathrm{SL}(2,\R) \to \mathrm{PSL}(n,\R)$ be a homomorphism coming from the $n$ dimensional irreducible representation of $\mathrm{SL}(2,\R)$. Fix a Fuchsian representation $\sigma: \pi_1(\Sigma_g) \to \mathrm{SL}(2,\R)$. The Hitchin component for $\mathrm{PSL}(n,\R)$ is the topological component of the representation variety $\mathrm{Rep}(\Sigma_g, \mathrm{PSL}(n,\R))$ containing $i_n \circ \sigma$, which we call $\textrm{Hit}(\Sigma_g,\mathrm{PSL}(n,\R))$ (more precisely, this is true when $n$ is odd, and when $n$ is even there are two isomorphic Hitchin components \cite{Hi}).

\begin{LabourieConjecture} 
Let $\rho$ be a representation in the Hitchin component for $\textrm{PSL}(n,\R).$ The energy functional $\mathbf{E}_\rho$ has a unique critical point.
\end{LabourieConjecture}
Shortly after making the conjecture, Labourie proved in \cite{L1} that $\mathbf{E}_\rho$ is proper on $\mathbf{T}_g$ when $\rho$ is well-displacing \cite[Definition 6.12]{L1}, and showed that Hitchin representations to $\textrm{PSL}(n,\R)$ are well-displacing. Hence, for $\rho$ in the Hitchin component for $\textrm{PSL}(n,\R),$ we can choose a Riemann surface structure that minimizes $\mathbf{E}_\rho$. Thus, the key point of the conjecture is uniqueness. We note that minimal maps for Hitchin representations are immersions (see Remark \ref{rem: immersion}).

Several versions of the Labourie Conjecture have appeared in the literature. In the first statement of the conjecture \cite{L0}, Labourie posited that $\mathbf{E}_\rho$ has a unique minimum. However, in later statements \cite{Labconvex}, \cite{L1}, and more generally in the literature, ``minimum" has been replaced by ``critical point" (for example see \cite[section 6.5]{Swoboda}), so this is what we take as the Labourie Conjecture. We also give a disproof of the ``minimum'' version in Corollary A. The critical point version is equivalent to the uniqueness of the equivariant minimal map, and to the uniqueness of an invariant minimal surface whose quotient is $\Sigma_g$. As well, in some sources the Labourie Conjecture is stated to be for all Hitchin representations into split real simple Lie groups (for instance, \cite[page 6]{CTT}, \cite[page 1]{M2}). Existence still holds: Guichard-Wienhard proved that all Anosov representations are well-displacing \cite[Theorem 1.7]{GuW}, and Guichard-Labourie-Wienhard proved that all Hitchin representations into all groups are Anosov \cite{GLW}. A generalized version of the conjecture that includes maximal representations is posed in \cite[section 7]{Li}.

Following Labourie (\cite[section 2]{L1} and \cite[section 1]{L2}), we briefly explain the construction of a mapping class group invariant complex structure on any Hitchin component for which the conjecture is true, using the theory of Higgs bundles. Let $S$ be a Riemann surface structure on $\Sigma_g$ with canonical bundle $\mathcal{K}$. Using Higgs bundles, Hitchin gives a parametrization of the Hitchin component for $\mathrm{PSL}(n,\R)$ by the Hitchin base
$$\mathbf{H}_{n}(S) =\oplus_{i=2}^n H^0(S,\mathcal{K}^i).$$ In $\mathbf{H}_{n}(S)$, Teichm{\"u}ller space $\mathbf{T}_g$ identifies non-holomorphically with the locus $H^0(S,\mathcal{K}^2)\times \{0\}\times \dots \times \{0\}.$ A drawback of Hitchin's parametrization is that it is not natural with respect to the mapping class group action on the Hitchin component; in particular the complex structure on $\mathbf{H}_n$ is not mapping class group invariant. To remedy this, Labourie proposes the following. For a Hitchin representation $\rho$ corresponding to differentials $(q_2, \ldots, q_n)\in \mathbf{H}_{n}(S)$, the equivariant harmonic map $h:\tilde{S}\to G/K$ is minimal if and only if $q_2=0$.
Varying over $\mathbf{T}_g$, it follows that the space of all equivariant minimal surfaces for Hitchin representations is parametrized by the holomorphic vector bundle $\mathbf{M}_{n}(\Sigma_g)$ over $\mathbf{T}_g$, whose fiber over a Riemann surface $S$ is $\oplus_{i=3}^n H^0(S,\mathcal{K}^i).$ The holonomy map $$L_{n}:\mathbf{M}_{n}(\Sigma_g)\to \textrm{Hit}(\Sigma_g,\textrm{PSL}(n,\R))$$ is mapping class group equivariant, and we refer to it as the Labourie map. Labourie's minimal surface existence result says that $L_{n}$ is surjective. We conclude that if the Labourie Conjecture holds for a given $n,$ $L_n$ is bijective, and by Brouwer's invariance of domain, $L_n$ is a homeomorphism.

For a general split real simple Lie group $G$ of adjoint type, the Hitchin component $\textrm{Hit}(\Sigma_g,G)\subset \textrm{Rep}(\Sigma,G)$ is defined analogously to $\textrm{Hit}(\Sigma_g,\textrm{PSL}(n,\R))$, using the principal embedding $\textrm{SL}(2,\R)\to G$ (see \cite{Hi}). For such $G$, there is a holomorphic vector bundle $\mathbf{M}_G(\Sigma_g)$ over $\mathbf{T}_g$ analogous to $\mathbf{M}_n(\Sigma_g)$ and a surjective mapping $L_G: \mathbf{M}_G(\Sigma_g)\to \textrm{Hit}(\Sigma_g,G)$, which is injective when uniqueness of minimal surfaces holds (see section \ref{sec: area minimizers}). For spaces of maximal representations in rank $2$, similar equivariant parametrizations of the space of minimal surfaces as complex manifolds have been obtained in \cite{Co} and \cite{AC}, although the structure of the complex manifolds is more complicated. 

We now survey some results toward the Generalized Labourie Conjecture. The uniqueness of the minimal surface for $\mathrm{PSL}(2,\R)$ follows more or less from the definition of a Fuchsian representation. Minimal surfaces for products of Fuchsian representations into $\textrm{PSL}(2,\R)\times \textrm{PSL}(2,\R)$ are the graphs of minimal Lagrangian diffeomorphisms of hyperbolic surfaces, which Schoen used to prove uniqueness \cite{Sc}. For the product of a Fuchsian representation and a non-Fuchsian representation, uniqueness is due to Wan \cite{Wa}. Labourie proved uniqueness for $\mathrm{PSL}(3,\R)$ using hyperbolic affine spheres \cite{Labconvex}. In this case, the vector bundle parametrization of the $\mathrm{PSL}(3,\R)$-Hitchin component had already been established by Loftin \cite{Lo}, via affine spheres. Labourie then proved uniqueness for the Hitchin components of all split real simple $G$ of rank $2$ \cite{L2} (this adds $\textrm{PSp}(4,\R)$ and $G_2$), using that the minimal surfaces lift to $J$-holomorphic curves tangent to a distribution in a homogeneous space of $G$. Finally, Collier-Tholozan-Toulisse proved uniqueness for maximal representations into Hermitian Lie groups of rank $2$ \cite{CTT}, which completes the list of known higher Teichm{\"u}ller spaces in rank $2$. See also the work of Collier \cite{Co}, Alessandrini-Collier \cite{AC}, and Nie \cite{Nie}. 

For applications of minimal surfaces to higher Teichm{\"u}ller theory beyond the Labourie Conjecture, see works such as \cite{B}, \cite{CTT}, \cite{LTW}.

\subsection{Main result}
We recall that a minimal map is one that is harmonic and (weakly) conformal. Our main result is the following theorem:

\begin{thma}
Let $G$ be the adjoint form of a split real semisimple Lie group with rank at least $3$ acting faithfully by isometries on a symmetric space $N$. For every $g\geq 3,$ there exists a Hitchin representation $\rho:\pi_1(\Sigma_g)\to G$ and an unstable $\rho$-equivariant minimal map into $N$. In particular, there are at least two equivariant minimal surfaces in $N$.
\end{thma}
For split real semisimple $G$ splitting into simple factors as $G=G_1\times \dots \times G_m,$ we use Hitchin to mean that each component representation to $G_i$ is Hitchin. In the semisimple case the associated symmetric space can have different invariant metrics leading to different meanings of ``minimal,'' and our theorem is true for any invariant metric. 

About the restriction to $g\geq 3,$ see Remark \ref{rem: genus 2}.

The unstable minimal surfaces of Theorem A can be described explicitly in terms of Higgs bundles. For example, for $G=\textrm{PSL}(4,\R)$, we have the following (see sections 2.3 and 4.1 for definitions).
\begin{thm}\label{thm: PSL(4,R)}
    For every $g\geq 3,$ there exist a closed Riemann surface $S$ of genus $g$ with canonical bundle $\mathcal{K}$ on which there exists abelian differentials $\phi_1,\phi_2,\phi_3,\phi_4\in H^0(S,\mathcal{K})$ such that, for real $R>0$ sufficiently large, the Higgs bundle in the $\textrm{PSL}(4,\R)$-Hitchin section associated to the Hitchin basepoint $$(q_2,q_3,q_4)=(0,\frac{1}{6}R^3(\phi_1\phi_2\phi_3+\phi_1\phi_2\phi_4+\phi_1\phi_3\phi_4+\phi_2\phi_3\phi_4),-\frac{2}{9}R^4\phi_1\phi_2\phi_3\phi_4)$$ determines a Hitchin representation $\rho:\pi_1(S)\to \textrm{PSL}(4,\R)$ together with a $\rho$-equivariant unstable minimal map to the symmetric space of $\textrm{PSL}(4,\R)$.
\end{thm}

In fact, in Theorem \ref{thm: PSL(4,R)} one can explicitly choose $S$ and the $\phi_i$'s. See the proofs of Theorem \ref{thm: PSL(4,R)} and Proposition \ref{exun} for details.

Theorem A shows that the Labourie map $L_G$ for $G$ is not injective. Hence, by Brouwer's invariance of domain, $L_G$ has no continuous section. Using this, together with a strengthened properness result for energy functionals (see the proof of Lemma \ref{lem: continuous}), we deduce:
\begin{cora}
For every $g\geq 3$ and pair $G,N$ as in Theorem A, there exists a Hitchin representation $\rho:\pi_1(\Sigma_g)\to G$ with at least two distinct area minimizing $\rho$-equivariant minimal maps into $N$.
\end{cora}
Although the maps $L_G$ are not bijective, there are other candidates for an equivariant and holomorphic parametrization of the Hitchin components. See \cite{FT} and \cite{No}.

Note that we've defined Hitchin representations for Lie groups of adjoint type, and hence stated our results for such Lie groups, but the condition is not really necessary since every Hitchin representation into the split real form of the adjoint group of a complex Lie group lifts to the split real form of the universal cover (see \cite{Hi}).

\begin{subsection}{The case $\textrm{PSL}(2,\R)^n$}\label{Rtrees}
To motivate the proof of Theorem A, we explain our previous work with Markovi{\'c} \cite{MSS}. On a fixed Riemann surface structure $S$ over $\Sigma_g$ with universal cover $\tilde{S}$, there is a one-to-one correspondence between holomorphic quadratic differentials on $S$ and pairs $(\rho,h),$ where $\rho$ is a Fuchsian representation to $\textrm{PSL}(2,\R)$ and $h$ is a $\rho$-equivariant harmonic map from $\tilde{S}$ to the hyperbolic space $\mathbb{H}^2$ (see \cite{Hi}, \cite{Wan}, and \cite{W2}). A holomorphic quadratic differential $\phi$ on $S$ also corresponds to a unique equivariant harmonic map to an $\mathbb{R}$-tree: in one direction, the leaf space $(T,d)$ for the vertical foliation of $\phi$ is an $\mathbb{R}$-tree, and the projection map from $\tilde{S}\to (T,d)$ is harmonic. Given $\phi$, we consider the ray of representations to $\textrm{PSL}(2,\R)$ with harmonic maps determined by $R\phi,$ $R>0.$ After suitably rescaling the metric on the target $\mathbb{H}^2$, the maps converge in an appropriate sense as $R\to\infty$ to a harmonic map to the $\mathbb{R}$-tree for $\phi$ (see \cite{W} for details).

Quadratic differentials $\phi_1,\dots, \phi_n$ sum to $0$ if and only if the equivariant map to $(\mathbb{H}^2)^n$ determined by $(R\phi_1,\dots, R\phi_n)$ is minimal. Upon rescaling, as $R\to\infty$, the minimal maps to $(\mathbb{H}^2)^n$ converge to an equivariant minimal map into a product of $\mathbb{R}$-trees. We proved in \cite[Theorem B2]{MSS} that unstable minimal surfaces for $\textrm{PSL}(2,\R)^n$ converge along such rays to unstable minimal surfaces in products of trees, and conversely every unstable minimal surface in a product of trees is approximated by unstable $\textrm{PSL}(2,\R)^n$-minimal surfaces. 

To produce unstable minimal maps to products of trees, we observed that when all the $\phi_i's$ are squares, each $\R$-tree folds onto a copy of $\R$, and the minimal map to the product projects onto a minimal map to $\R^n$ that is equivariant for a representation of $\pi_1(\Sigma_g)$ to the additive group $(\R^n,+)$. We proved that instability in these two contexts is equivalent \cite[Theorem C]{MSS}. We then noted that unstable equivariant minimal surfaces in $\R^n$ are abundant for $n\geq 3$. Indeed, the most natural example is the lift of an unstable minimal surface in the $n$-torus to $\R^n$, and, in fact, any non-planar equivariant minimal surface in $\R^3$ is unstable (see \cite[section 5.3]{MSS}). 

In the general setting where we start with $n$ quadratic differentials summing to $0$, giving an equivariant minimal map into a product of trees, there is a branched covering on which the differentials lift to squares. If the corresponding equivariant minimal map from the branched cover into $\R^n$ can be destabilized through equivariant variations that are invariant under the Deck group of the branched covering, then the minimal surface in the original product of $\mathbb{R}$-trees is unstable.

\end{subsection}

\begin{subsection}{Hitchin rays and the idea of the proof}
Here we explain how to generalize the ideas of section \ref{Rtrees} to other symmetric spaces using Higgs bundles in place of Hopf differentials, and we give the strategy for Theorem A. Our overview here is non-technical, and more precise definitions are given later in the text.

Let $G$ be a  real semisimple Lie group of real rank $r$ with finite center and no compact factors. Parreau has compactified $\textrm{Rep}(\Sigma_g,G)$, with boundary objects corresponding to $\pi_1(\Sigma_g)$-actions on rank $r$ buildings \cite{Pa}. Hitchin rays, defined below, generalizes the rays $R\phi$ above. A Hitchin ray on a Riemann surface $S$ associated with a non-nilpotent $G$-Higgs bundle (defined in section \ref{1.4}) determines a ray of pairs $(\rho_R,h_R),$ $R>0,$ where $\rho_R:\pi_1(\Sigma_g)\to G$ is a representation and $h_R$ is a $\rho_R$-equivariant harmonic map from $\tilde{S}$ to a symmetric space of $G$, such that the energy is blowing up: $$\lim_{R\to\infty}\mathcal{E}(S,h_R)=\infty.$$ For large $R,$ we refer to $h_R$ as a high energy harmonic map, or a high energy minimal map if it is minimal. As pointed out in \cite[section 1.3]{KNPS}, it follows from work of Kleiner-Leeb \cite{KL}, Korevaar-Schoen \cite{KS}, and now-routine estimates that along Hitchin rays, the harmonic maps converge in an appropriate sense to a harmonic map to a building of rank $r$. Our previous work \cite{MSS} thus suggests that one could find a counterexample to the Labourie Conjecture by first finding an unstable minimal surface in a building.  It's then feasible to imagine that these minimal surfaces in buildings behave like minimal surfaces in $\R^r$. Perhaps after passing to a branched covering, the apartments of the buildings can fold onto each other in a way that produces equivariant minimal surfaces in $\R^r$. This geometric picture gives an intuitive reason to expect stability in rank $2$ and instability in rank at least $3$. The difficulty in turning this picture into a proof is that the precise nature of convergence and the behaviour of limiting harmonic maps to buildings is poorly understood in comparison to harmonic maps to $\R$-trees. While we don't know if this reasoning can lead to a proof of stability for rank $2$ Hitchin representations, this picture indeed provides the key intuition for producing unstable minimal surfaces in rank at least $3$.

We now describe the Hitchin rays in more detail for $G$ isogenous to $\mathrm{PGL}(n,\C)$, in which case $G$-Higgs bundles over a Riemann surface give rise to classical Higgs bundles as defined by Hitchin. See section 2.3 for all the definitions. Note that the rank of $\mathrm{PGL}(n,\C)$ is $n-1$. Let $S$ be a Riemann surface structure on $\Sigma_g$ with canonical bundle $\mathcal{K}$. A Higgs bundle $(E,\overline{\partial}_E, \phi)$ of rank $n$ is a holomorphic vector bundle $(E,\overline{\partial}_E)$ of rank $n$ over $S$ with a holomorphic section $\phi \in H^0(S,\textrm{End}(E) \otimes \mathcal{K})$ called the Higgs field. By the non-abelian Hodge correspondence, every pair consisting of a representation to $\textrm{PGL}(n,\C)$ and equivariant harmonic map to the $\textrm{PGL}(n,\C)$-symmetric space (including harmonic maps to the $\mathrm{PGL}(n,\R)$ subspace) arises from a unique polystable traceless Higgs bundle $(E,\overline{\partial}_E, \phi)$ of rank $n$, up to tensoring $E$ with a holomorphic line bundle, and vice versa. By scaling the Higgs field by $R>0$, we obtain a family of representations and harmonic maps that is called the Hitchin ray. Although the Higgs bundle is a tractable holomorphic object, the harmonic map is not easy to understand from the Higgs bundle data. Indeed,
the harmonic map is related to the Higgs bundle through Hitchin's self-duality equations, which form a system of $n^2$ coupled non-linear PDE's. 

Every Higgs bundle has an associated branched covering space called the (reduced, normalized) cameral cover (\cite[section 3]{KNPS}, \cite{Don}). We call any connected component a small cameral cover (Definition \ref{def: smallcameral}). Each small cameral cover comes equipped with an equivariant harmonic map $\tilde{f}$ from its universal cover to $\R^{n-1}$ (for the precise notion of equivariance see section 3.2). Up to isomorphism, both the cover and the map depend only on the point in the Hitchin base associated with $(E,\overline{\partial}_E,\phi)$ (see section \ref{sec: hitchin section}).

The Higgs bundle $(E,\overline{\partial}_E,\phi)$ is called generically semisimple if the Higgs field is semisimple on the complement of a discrete set. Fortunately, for such Higgs bundles it is possible to use maximum principle arguments to understand the local behavior of Hitchin's equations away from the branch locus of the cameral covering as $R\to \infty$. This is carried out by Mochizuki, who calls the resulting phenomenon asymptotic decoupling \cite{Mo}. 

Let $(E,\overline{\partial}_E,\phi)$ be a generically semisimple stable traceless Higgs bundle with $\textrm{tr}(\phi^2)=0$; the last condition is equivalent to the minimality of $\tilde{f}$ and every harmonic map along the Hitchin ray. Thanks to the equivariance properties, the intrinsic data (the induced metric and part of the second fundamental form) of the minimal map to the symmetric space associated to the Higgs bundle, as well as that of the map $\tilde{f}$ to $\R^{n-1}$, descend all the way to $S$, where they can be compared. We extrapolate from Mochizuki's analysis that off the branch locus, the intrinsic data of the minimal maps associated with the Higgs bundles $(E,\overline{\partial}_E,R\phi)$ converges to the intrinsic data of the map $\tilde{f}$. Unfortunately, the limiting behaviour of harmonic maps at the branch locus is yet to be understood. The local picture at these points has been worked out only for minimal surfaces associated to Lie groups of rank $2$ \cite{DW}, \cite{TW}, \cite{Evans}, \cite{LTW}.

Happily, when considering variations of mimimal surfaces, the difficulties at the branch locus are readily overcome: by the log cut-off trick, any destabilizing variation of a minimal surface can be perturbed to a destabilizing variation supported on the complement of any discrete set. The upshot of Mochizuki's work and this last observation is that we can maneuver around the technicalities of buildings, and treat high energy minimal maps as if they approximate minimal surfaces in $\R^{n-1}$. With this guiding principle, we prove in Theorem B that if the limiting minimal surface in $\R^{n-1}$ is unstable under variations invariant by a certain action, then so are the high energy minimal surfaces in the symmetric space.
This instability result is the technical step, although the computation becomes simple and natural if we work in the setting of harmonic bundles and vary by well-chosen gauge transformations.

Finally, to prove Theorem A for $\textrm{PSL}(n,\R)$, we take the following strategy. We begin with a closed Riemann surface $S$ and a pair $(\tilde{f},\chi)$ consisting of a representation $\chi:\pi_1(\Sigma_g)\to (\R^{n-1},+)$ and an unstable $\chi$-equivariant minimal map $\tilde{f}=(\tilde{f}_1,\dots, \tilde{f}_{n-1}): \tilde{S}\to \R^{n-1}$. It is possible to choose $\tilde{f}$ to be unstable if and only if $n\geq 4$. For every $i$, the $(1,0)$-part of the derivative of $\tilde{f}_i$ descends to a holomorphic $1$-form $\phi_i$ on $S$. We call the $\phi_i$'s the Weierstrass data of $\tilde{f}$. Minimality gives that the squares of the $\phi_i$'s sum to zero. This Weierstrass data can be assembled into a Cartan-valued holomorphic $1$-form for $\mathfrak{sl}(n,\C)$, and evaluating adjoint-invariant polynomials on this form defines a point in the Hitchin base (see sections \ref{sec: hitchin section} and \ref{sec: unstable minimal surfaces}). Applying the Hitchin section to this point returns a traceless Higgs bundle $(E,\overline{\partial}_E,\phi)$ such that $\textrm{tr}(\phi^2)=0$. Since the Hitchin component is connected, applying the non-abelian Hodge correspondence along the Hitchin ray yields a ray of Hitchin representations. Every component of the (reduced, normalized) cameral cover identifies with $S$, and the equivariant minimal map from the universal cover of a component is the composition of $\tilde{f}$ with an isometry. Moreover, the high energy minimal maps locally approximate the original $\tilde{f}$, up to an isometry. Theorem B then shows that the high energy minimal maps are unstable, which proves Theorem A. The Hitchin section is defined for all split real $G$, and we use this to prove Theorem A in general. 
\end{subsection}

\begin{subsection}{The index of minimal maps}\label{1.4} 
 We now give a careful statement of Theorem B. To keep things classical in the introduction, we continue to focus mainly on the $\mathrm{PGL}(n,\C)$ case.

Let $(E,\overline{\partial}_E,\phi)$ be a generically semisimple stable traceless Higgs bundle. There exists $m \leq n$ and a finite subset $B$ such that for every $p\in S-B$, there is a neighbourhood $U\subset S-B$ of $p$ on which the Higgs bundle decomposes holomorphically as $E|_U=\oplus_{i=1}^m E_i$, and $\phi$ acts on each $E_i$ by multiplication by a holomorphic $1$-form $\phi_i$. Ranging over a covering of $S-B$, the locally defined projections to the $E_i$ span a holomorphic subbundle $F_\phi\subset\textrm{End}(E)|_{S-B}$ that we call the extended toral bundle of $\phi$. The extended toral bundle carries a flat metric, flat connection, and a real structure such that the projections are flat, real, and orthogonal (see section 3.2 for details). We call the bundle of traceless endomorphisms in $F_\phi$ the toral bundle, $F_\phi^0$, and the Higgs field $\phi$ is a holomorphic $F_\phi^0$-valued $1$-form. Using the real structure on $F_\phi$, the real part of $\phi$ further defines a flat Riemannian affine bundle $M_\phi$ over $S$, which we call the apartment bundle, and which comes equipped with a harmonic section $f$ whose holomorphic derivative is $\phi$. This is all constructed formally in section \ref{sec: 3.2}.

There is a Galois branched covering $\tau: C\to S$ on which the monodromy of $F_\phi$ trivializes, and this is an example of a small cameral cover (Definition \ref{def: smallcameral}). On $C-\tau^{-1}(B)$, $\phi$ admits globally defined eigen-$1$-forms $\phi_1,\dots, \phi_m$, which then extend to all of $C$ by the removable singularities theorem. Lifting to a covering space $\tilde{C}$ and integrating the real parts 
gives a harmonic map to $\R^{m}$ 
$$
\tilde{f} = (\tilde{f}_1,\dots,\tilde{f}_m), \hspace{1mm} \tilde{f}_i(z)= \sqrt{r_i} \textrm{Re}\int_{z_0}^z \phi_i,
$$
where $r_i = \mathrm{rank}(E_i)$, unique up to translation and equivariant with respect to a representation to $(\R^m,+)$, and which transforms in a certain when we act by the lift of a Deck transformation of $\tau:C\to S$.
It's tidier to work with the equivalent harmonic section of the apartment bundle $M_\phi$, but it's helpful to keep this geometric picture in mind.

Using Hitchin and Simpson's side of the non-abelian Hodge correspondence, the Higgs bundle produces a surface group representation $\rho:\pi_1(\Sigma_g)\to \textrm{PGL}(n,\mathbb{C})$ and a $\rho$-equivariant harmonic map $h$ from $\tilde{S}$ to the symmetric space of $\textrm{PGL}(n,\mathbb{C})$. This is equivalent to the data of a bundle $E$ with projectively flat connection $D$ and harmonic metric $H$ unique up to scale. If we rescale the Higgs field via the $\mathbb{R}^+$-action $R\cdot (E,\overline{\partial}_E,\phi) = (E,\overline{\partial}_E,R\phi)$, then we get new pairs of representations and maps $(\rho_R,h_R)$, or equivalently triples $(E_R,D_R,H_R)$. While our results concern the maps $h_R,$ for the problem at hand it's easier to compute with the metrics $H_R$.

The index of a minimal section of a flat Riemannian bundle (or equivalently an equivariant minimal map) is the maximal dimension of a space of variations on which the second variation of area is negative definite (see section 3.1). A minimal surface is stable if the index is zero, and unstable otherwise. For each $R$, the harmonic map $h_R$ to the symmetric space is minimal if and only if the harmonic section $f$ is minimal (if and only if the harmonic map $\tilde{f}$ to $\R^m$ is minimal). Assuming minimality of $f$, let $\textrm{Ind}(H_R)$ be the equivariant index for the area of $h_R$, and $\textrm{Ind}(f)$ the index of $f$ as a section of $M_\phi$, which is the same as the equivariant index for the area of $\tilde{f}$. These indices are rigorously defined in section \ref{sec: 3.2}. The index $\textrm{Ind}(H_R)$ agrees with the index of the energy functional on Teichm{\"u}ller space associated to the holonomy $\rho_R$ \cite[Theorem 3.4]{Ej}.

\begin{thmbp}
$\liminf_{R\to\infty} \textrm{Ind}(H_R) \geq \textrm{Ind}(f)$.
\end{thmbp}
In particular, if $f$ is unstable, then so is $h_R$ for sufficiently large $R$. 

We turn now to the general case of a semisimple Lie group $G$ acting on a symmetric space $N$, and a generically semisimple stable $G$-Higgs bundle $(P, A^{0,1}, \phi)$. Here $(P, A^{0,1})$ is a holomorphic $K^\C$-bundle and $\phi$ is a holomorphic $1$-form valued in $P \times_{K^\C} \p^\C$. As in the classical case, there is an $\R^+$-family of representations and harmonic maps, or $G$-harmonic bundles $(M_R,h_R)$, in which $M_R$ is the bundle with fiber $N$ associated to $P$, which inherits a flat connection, and $h_R$ is a harmonic section. 

The right generalization of the toral bundle $F^0_\phi$ to this setting is what we call the $G$-toral bundle of $\phi$, whose fiber is the center of the centralizer of $\phi$ in $P \times_{K^\C}\p^\C$. As in the $\mathrm{PGL}(n,\C)$ case, this is really a vector bundle away from some bad set $B$, and its rank is at most the rank of $G$. From the toral bundle, we build the $G$-apartment bundle $M^G_\phi$, together with its harmonic section $f$ (section \ref{sec: G apartment}). In the case where $f$ is minimal (equivalently, every harmonic section $h_R$ is minimal), we prove

\begin{thmb}
$\liminf_{R\to\infty} \textrm{Ind}(h_R) \geq \textrm{Ind}(f)$.
\end{thmb}

The proof almost follows directly from the case of $\mathrm{PGL}(n,\C)$ by using the adjoint representation of $G$. One minor complication arises from our insistence, in the semisimple case, of allowing an arbitrary invariant metric on the symmetric space $N$.

\begin{remark}
By the work of Toledo \cite{Tol}, the associated energy functional on Teichm{\"u}ller space is plurisubharmonic. This implies that $\textrm{Ind}(h_R)$ is at most $3g-3$. 
\end{remark}
\begin{remark}\label{rem: index behaviour}
   Theorems B2 and C from \cite{MSS} together yield that for products of Fuchsian representations into $G=\textrm{PSL}(2,\R)^n$, $\textrm{Ind}(h_R)$ is non-decreasing with $R$ and limits to $\textrm{Ind}(f).$ The proof of Theorem B2 relies on the Reich-Strebel formula, which is a tool particular to $\textrm{PSL}(2,\R)^n$. For general Higgs bundles, it would be very interesting to study how $\textrm{Ind}(h_R)$ changes with $R$. If $\textrm{Ind}(h_R)$ increases to $\textrm{Ind}(f)$ for Hitchin representations, then this would give another proof of the Labourie Conjecture in rank $2$.
\end{remark}

\begin{remark}
    The key to deducing Theorem A from Theorem B is that the Hitchin map (usually called the Hitchin fibration, see section \ref{sec: hitchin section}) restricted to the Hitchin component is surjective. If $\mathcal{H}$ is any other higher Teichm{\"u}ller space with the same property, then the proof of Theorem A goes through for $\mathcal{H}$. For example, this holds for Collier's components for the split real group $\textrm{SO}(n,n+1)$ \cite{Brian}. It is likely that this property or a slightly weaker version of it holds for all known higher Teichm{\"u}ller spaces (see \cite{BCGGO}). 
\end{remark}

\end{subsection}

\begin{subsection}{Acknowledgements}
We would like to thank Nicolas Tholozan, Eugen Rogozinnikov, Max Riestenberg, Brian Collier, and Jérémy Toulisse for helpful conversations, and Mike Wolf in particular for pointing out the connection between \cite{M2} and minimal surfaces in $\R^n$. We also thank the referee for many suggestions that significantly improved the paper. Nathaniel Sagman is funded by the FNR grant O20/14766753, \textit{Convex Surfaces in Hyperbolic Geometry.}
\end{subsection}
\end{section}

\begin{section}{Preliminaries}
\begin{subsection}{Hermitian metrics} 
For a complex vector space $V$, we write $V^{\vee}$ for its dual and $\overline{V}$ for its complex conjugate. Similarly, for a linear map of complex vector spaces $f:V\to W,$ we write $f^\vee:W^\vee \to V^\vee$ for the transpose and $\overline{f}:\overline{V}\to\overline{W}$ for the conjugate. A Hermitian metric $H$ on $V$ is a linear isomorphism $H: V\to \overline{V}^\vee$ such that $H=\overline{H}^\vee$ and the associated sesquilinear form is positive definite. If $f\in \textrm{End}(V)$ and $H$ is a Hermitian metric on $V$, then the adjoint with respect to $H$ is $f^{*_H}=H^{-1}\overline{f}^\vee H:V\to V$.

Let $\mathrm{Met}(V)$ be the space of Hermitian metrics on $V$ and $\Met(V)$ its quotient by positive real scaling. We will use $[H]$ to denote the equivalence class of a metric $H$ in $\Met(V)$. The group $GL(V)$ acts transitively on $\mathrm{Met}(V)$ via 
\begin{equation} \label{eqn: action}
g\cdot H = {\overline{g}^{\vee}}^{-1} H g^{-1} 
\end{equation}
and this descends to a transitive action of $G = \mathrm{PGL}(V)$ on $\Met(V)$. The stabilizer of a class $[H] \in \Met(V)$ is a maximal compact subgroup $K_H\subset G$.

Let $\mathfrak{g}$ be the Lie algebra of $G$, which we identify with the traceless endomorphisms $\textrm{End}^0(V)$. The space $\mathrm{Met}(V)$ is an open subset of the real vector space $\mathrm{Hom}(V, \overline{V}^\vee)_\R$ of maps equal to their conjugate transpose. If $Y \in \textrm{End}(V)$, write $Y^0$ for its traceless part. The $\g$-valued $1$-form $\varphi$ on $\mathrm{Met}(V)$ defined by
\begin{equation} \label{eqn: varphi}
   \varphi_{H} = -\frac{1}{2}(H^{-1}dH)^0 
\end{equation}
is basic with respect to scaling, and hence descends to a $1$-form on $\Met{V}$ that we also call $\varphi$. If $X \in \mathfrak{g}$, then the derivative of the action $\eqref{eqn: action}$ on $\mathrm{Met}(V)$ is given by
\[
X \cdot H = -\overline{X}^\vee H - H X.
\]
Combining the last two equations, we find
\begin{equation} \label{eqn: varphi of action}
  \varphi_{[H]}(X \cdot [H]) = \frac{1}{2}(X^{*_H} + X).
\end{equation}
Note that the operator $*_H$ depends only on the projective equivalence class $[H]$.

Let $\mathrm{End}_H^0\subset \mathfrak{g}$ be the space of $H$-self-adjoint endomorphisms of $V$. Then $\varphi_{[H]}(X \cdot [H])$ projects $X$ to $\mathrm{End}_H^0$, and in fact $\varphi$ defines an isomorphism from $T\Met(V)$ to the bundle over $\Met(V)$ whose fiber over a metric $[H]$ is $\mathrm{End}_H^0$.

The Killing metric on $\Met(V)$ is defined to be \begin{equation} \label{eqn: killing}
2n\tr(\varphi\otimes \varphi).
\end{equation}
In the sequel, we will omit the tensor product $\otimes$ from our notation, writing this as $2n \tr(\varphi^2)$, which is to be read as a bilinear form. This metric is $G$-invariant, and exhibits $\Met(V)$ as a symmetric space of non-compact type. Every $G$-invariant metric on $\Met(V)$ is a constant multiple of the Killing one.

To simplify the statement of the following proposition, we use $\varphi$ to identify $T\Met(V)$ with the subbundle $\mathrm{End}_H^0$ of the trivial $\g$ bundle over $\Met(V)$ whose fiber over $[H]$ is the subspace $\mathrm{End}_H^0$. The Lie bracket and differential $d$ are defined on this trivial $\g$-bundle.

\begin{prop}[Theorem X.2.6 in \cite{KN}] \label{prop: grad and R} For sections $X,Y,Z$ of $\mathrm{End}_H^0$, implicitly identified with vector fields on the symmetric space, the Levi-Civita connection $\nabla^{lc}$ is given by
\begin{equation} \label{eqn: nabla lc}
    \nabla_X^{lc} Y= d_X Y - [X,Y]
\end{equation}
and the curvature of $\nabla^{lc}$ is given by
\[
R(X,Y)Z = -[[X,Y],Z].
\]
\end{prop}

\begin{remark}
Although the metric is canonical only up to a constant, the Levi-Civita connection and curvature are independent of this choice.
\end{remark}

\end{subsection}

\begin{subsection}{Harmonic maps, bundles, and minimal surfaces}
Let $S$ be a possibly non-compact Riemann surface and $h: S \to N$ a smooth map from $S$ to a Riemannian manifold $N$. Let $\mathcal{K}$ be the canonical bundle of $S$, and let $\partial h$ be the $(1,0)$-part of the derivative of $h$, which is a section of $h^*TN \otimes_\R \mathcal{K}$. Let $\dbar$ be the $(0,1)$-part of the connection on $h^*TN \otimes_\R \C$ induced from the Levi-Civita connection of $N$, as well as its natural extension to $h^*TN\otimes_\R \C$-valued forms.

\begin{defn} \label{dfn: harmonic}
The map $h$ is harmonic if $\overline{\partial}\partial h = 0$. 
\end{defn}

If $\nu = \langle \cdot, \cdot \rangle$ is the Riemannian metric on $N$, the Hopf differential of $h$ is the section of $\mathcal{K}^2$ defined by $\langle \partial h, \partial h \rangle$. We call $h$ conformal if $\langle \partial h, \partial h \rangle = 0$. We eschew words like almost conformal and branched conformal, instead calling a classically conformal map a conformal immersion. 

\begin{defn} \label{dfn: minimal}
A harmonic and conformal map is called a minimal map.
\end{defn}

The area density of $h$ is the non-negative $(1,1)$-form on $S$ given by
\begin{equation}\label{adens}
    a_h = \sqrt{\langle \partial h, \overline{\partial} h \rangle ^2 - |\langle \partial h, \partial h \rangle |^2},
\end{equation}
where the expression inside the square-root is interpreted as a real section of $\mathcal{K}^2 \otimes \overline{\mathcal{K}}^2$. If $S$ is closed, we write the area of $h(S)$ as 
\begin{equation}\label{ar}
    A(h) = \int_S a_h.
\end{equation}
Analogously, we define the energy density $e_h = \langle \partial h, \dbar h \rangle$ and total energy $\mathcal{E}(S,h) = \int_S e_h$. For arbitrary $S$, $h$ is harmonic if and only if it is a critical point of $\mathcal{E}(S,\cdot)$ with respect to compactly supported variations on all compact subsurfaces, and if $h$ is harmonic, then it is minimal if and only if it is a critical point of $A$ in this sense. In particular, a minimal map defines a minimal surface in the classical sense if it is an embedding. 

\begin{remark}
A minimal map $h$ is allowed to be constant. If it is non-constant, then since $\partial h$ is a holomorphic section of a holomorphic vector bundle, its vanishing locus is discrete. By conformality, $h$ is an immersion wherever $\partial h$ is nonzero. Therefore, a nonconstant minimal map is an immersion away from a discrete set. 
\end{remark}

\begin{subsubsection}{Flat Riemannian bundles} \label{sec: flat Riemannian}
Most of the time, for us $S$ will be a closed Riemann surface. If $N$ is a Riemannian manifold, $\tilde{S}$ is a universal covering of $S$, and $\rho: \pi_1(S) \to \mathrm{Isom}(N)$ is an action by isometries, then a $\rho$-equivariant map from $\tilde{S}$ to $N$ is the same as a section of the bundle $\tilde{S} \times_\rho N = \{(p,n) \in \tilde{S} \times N / (p,n) \sim (p \gamma, \rho(\gamma^{-1}) n)\}$. This bundle carries the structure of a flat Riemannian bundle, by which we mean a fiber bundle together with a smooth fiberwise Riemannian metric and a flat isometric connection. We think of the connection as an Ehresmann connection (see \cite{Eh}). Conversely, every flat Riemannian bundle $M \to S$ can be written in the form $\tilde{S} \times_\rho N$, since upon fixing a universal cover of $S$, one can recover $N$ as the space of flat sections of the pullback $\tilde{S} \times_S M$. Following Donaldson \cite{D}, we will often find it more convenient to work with sections of flat Riemannian bundles rather than with equivariant maps.

Let $M \to S$ be a flat Riemannian bundle, and let $h$ be a section. We say that $h$ is minimal (or harmonic or conformal respectively) if the corresponding equivariant map from a universal cover of $S$ is minimal (resp. harmonic, conformal). 

We can also characterize minimality directly using the vertical tangent bundle $T^\mathrm{vert}M$ of $M$. This bundle comes equipped with a Riemannian metric and a connection lifted from the Levi-Civita connection on the fibers of $M$ using the integrable foliation transverse to the fibers. If $h$ is a section of $M$, we interpret $\partial h$ as a $1$-form on $S$ valued in $h^*T^\mathrm{vert}M$ and $\dbar$ in terms of the connection on $T^\mathrm{vert}M$. Having done so, the original definitions of harmonic and conformal still make sense. Naturally, the formulas (\ref{adens}) and (\ref{ar}) still define the area density and area for sections of flat Riemannian bundles, and minimal sections are still critical points of area. The condition that $h$ is an immersion is potentially ambiguous; the correct condition is that the area form $a_h$ (defined with respect to the vertical tangent bundle) is nonvanishing. It remains true that a minimal section is an immersion away from a discrete set.

\end{subsubsection}

\begin{subsubsection}{Harmonic and minimal bundles}
Let $E$ be a complex vector bundle over $S$. A connection $D$ on $E$ is called projectively flat if its curvature is equal to $\omega \mathrm{Id}$, where $\omega$ is a smooth $2$-form. Since the identity is in the center of $\mathrm{End}(E)$, a projectively flat connection on $E$ induces a flat connection on $\mathrm{End}(E)$.

If $(E,D)$ is a projectively flat bundle, we define the projectivized bundle of metrics $\Met(E)$, whose fiber at a point $z$ is $\Met(E(z))$. A choice of invariant metric $\nu$ on $\Met(E)$ makes it into a flat Riemannian bundle. Unless we specify $\nu$, we will always take the Killing metric (\ref{eqn: killing}). The projectively flat connection $D$ defines a flat connection on $\Met(E)$. In this way, $\Met(E)$ has the structure of a flat Riemannain bundle. 

The main object of study in this paper will be triples $(E, D, H)$, where $E$ is a vector bundle over $S$, $D$ a projectively flat connection, and $H$ a smooth Hermitian metric. Given such a triple, let $[H]$ be the corresponding section of $\Met(E)$. The first part of the following definition is standard, but the second part is not.

\begin{defn}
$(E,D,H)$ is a harmonic bundle, and $H$ is a harmonic metric, if the associated section $[H]$ of $\Met(E)$ is harmonic. If $[H]$ is moreover minimal, we call $(E,D,H)$ a minimal harmonic bundle.
\end{defn}
Note that the definition of minimal would be equivalent had we chosen any other scaling of the metric on $\Met(E)$, and that the definition of harmonic does not reference the metric at all, only the Levi-Civita connection.

In the introduction, we mentioned the existence theorem for harmonic maps of Donaldson \cite{D} and Corlette \cite{C}. Rephrasing in terms of harmonic metrics, this gives one side of the non-abelian Hodge correspondence. 
A representation $\rho:\pi_1(\Sigma_g)\to\textrm{PGL}(n,\mathbb{C})$ is called reductive if its Zariski closure is a reductive subgroup of $\textrm{PGL}(n,\mathbb{C})$. The representation is irreducible if its holonomy is not contained in a non-trivial parabolic subgroup of $\textrm{PGL}(n,\mathbb{C})$. Here we're using irreducible in the sense that it defines a smooth point in the character variety. 
\begin{thm}[Non-abelian Hodge correspondence I]\label{NAHI}
Let $(E,D)$ be a projectively flat vector bundle of rank $n$ over a closed surface $S$, and suppose that the holonomy representation is reductive. Then there exists a harmonic metric $H$ on $(E,D)$. If the holonomy is irreducible, then $[H]$ is unique.
\end{thm}
Throughout the paper, we use the abbreviation NAH for the non-abelian Hodge correspondence. We will often implicitly choose a representative $H$ of $[H]$ and write that $(E,D,H)$ is determined by NAH I, even though only $[H]$ is uniquely defined. Since we only ever use the operator $*_H$, which depends only on $[H]$, we hope the reader will excuse this laziness.

\end{subsubsection}

\end{subsection}

\begin{subsection}{Higgs bundles}\label{2.3}
We owe an interpretation of the Definitions \ref{dfn: harmonic} and \ref{dfn: minimal} of harmonic and minimal in the context of triples $(E,D,H)$. First we need a good description of $[H]^* T^\mathrm{vert}\Met(E)$. This is provided by the endomorphism-valued $1$-form $\varphi$ defined in \eqref{eqn: varphi}. On each fiber $E(z)$, $\varphi$ is an $\mathrm{End}^0(E(z))$-valued $1$-form on $\Met(E(z))$, which for each $[H'] \in \Met(E(z))$, identifies $T_{[H']} \Met(E(z))$ with $\mathrm{End}^0_{H'}(E(z))$, the space of $H'$-self-adjoint endomoprhisms of $E(z)$. Note the complexification of $\mathrm{End}^0_{H'}(E)$ is just $\mathrm{End}^0(E)$, so $\varphi$ identifies the complexification of the tangent space at each point $[H']$ with $\mathrm{End}^0(E)$.

In particular, if $(E,D,H)$ is a flat bundle with a metric, and $[H]$ is the corresponding section of $\Met(E)$, then the pullback $\varphi_H := [H]^*\varphi$ of $\varphi$ by the section $[H]$ identifies $[H]^*T^{\mathrm{vert}}\Met(E)$ with $\mathrm{End}^0_H(E)$, and therefore identifies the space of variations of the section $[H]$ with the space of sections of $\mathrm{End}^0_H(E)$ over $S$. 
Since $\varphi_H$ is an $\mathrm{End}(E)$-valued $1$-form on $S$, we can define a connection on $E$ by
\begin{equation}\label{eqn:nablaH}
\nabla^H = D - \varphi_H.
\end{equation}
This induces the connection $D - [\varphi_H, \cdot]$ on $\mathrm{End}(E)$, which we also write as $\nabla^H$. If $X$ and $Y$ are sections of $\mathrm{End}_H^0(E)$ (which we implicitly identify with $[H]^*T^{\mathrm{vert}}(M)$), then using equation \eqref{eqn: varphi of action} and the fact that $X$ is already equal to its $H$-self-adjoint part, we have
\[
\nabla^H_X Y = D_X Y - [X,Y].
\]
Comparing this with the formula \eqref{eqn: nabla lc} in Proposition \ref{prop: grad and R}, we conclude
\begin{prop} \label{prop: connections agree}
The subbundle $End^0_H(E)$ of traceless $H$-self-adjoint endomorphisms is invariant by $\nabla^H$, and the restriction of $\nabla^H$ to this bundle is the pullback by the section $[H]$ of the Levi-Civita connection on $T^\mathrm{vert}\Met(E)$.
\end{prop}

Proposition \ref{prop: connections agree} allows us to express harmonicity of $(E,D,H)$ in terms of the connection $\nabla^H$. Using the Riemann surface structure on $S$, let $\phi$ be the $(1,0)$-part of $\varphi_H$, and let $\dbar_E$ be the $(0,1)$-part of the connection $\nabla^H$. Recall also that the Killing metric on $\mathrm{End}^0_H(E)$ is given by $\langle X, Y \rangle = 2n\tr(XY)$. Definitions \ref{dfn: harmonic} and \ref{dfn: minimal} translate to the result below.

\begin{prop}
$(E,D,H)$ is a harmonic bundle if $\dbar_E \phi = 0$. It is a minimal bundle if in addition $\tr(\phi^2) = 0$.
\end{prop}
Since we're working on a Riemann surface, the Koszul-Malgrange theorem guarantees the existence of a complex structure on $E$ with del-bar operator $\overline{\partial}_E$. This motivates the introduction of Higgs bundles.
\begin{defn}
A Higgs bundle on $S$ is the data $(E,\overline{\partial}_E,\phi),$ where $(E,\overline{\partial}_E)$ is a holomorphic vector bundle of rank $n$ and $\phi$ is a holomorphic section of $\textrm{End}(E)\otimes \mathcal{K}$ called the Higgs field. We call the Higgs bundle traceless if $\textrm{tr}(\phi)=0.$
\end{defn}
In this context, Theorem \ref{NAHI} says that every projectively flat bundle $(E,D)$ with reductive holonomy gives the data of a traceless Higgs bundle whose holomorphic structure is given by the $(0,1)$ part of \eqref{eqn:nablaH}. The second part of the non-abelian Hodge correspondence, Theorem \ref{NAHII} below, gives a converse. Assume now that $S$ is a closed Riemann surface. If $\deg(E)$ is the degree of $E$ and $\textrm{rank}(E)$ the rank, the slope of a complex vector bundle $E$ over $S$ is the quantity $$\mu(E) = \frac{\textrm{deg}(E)}{\textrm{rank}(E)}.$$
\begin{defn}
A Higgs bundle $(E,\overline{\partial}_E,\phi)$ is stable if for all proper $\Phi$-invariant holomorphic subbundles $E'\subset E$ of positive rank, we have $\mu(E')<\mu(E)$. It is polystable if it is a direct sum of stable Higgs bundles all of the same slope.
\end{defn}
We sometimes write slope-stable so as to not cause confusion with stability for minimal surfaces. The proposition below is clear from the definitions, and we record it for future use.
\begin{prop}\label{scale}
If $(E,\overline{\partial}_E,\phi)$ is (poly)-stable, then for $\gamma \in\mathbb{C}^*,$ $(E,\overline{\partial}_E,\gamma\phi)$ is (poly)-stable.
\end{prop}
The converse to Theorem \ref{NAHI} is due to Hitchin \cite{Hi} and Simpson \cite{Si}. Given a Hermitian metric $H$ on a holomorphic bundle $(E,\overline{\partial}_E)$, let $\partial^H$ be the unique $(1,0)$-connection with the property that $\overline{\partial}_E+\partial^H$ is the Chern connection for $H$. 
\begin{thm}[Non-abelian Hodge correspondence II]\label{NAHII}
If a Higgs bundle $(E,\overline{\partial}_E,\phi)$ is polystable, then there exists a Hermitian metric $H$ on $E$ such that $(E,\overline{\partial}_E+\partial^H + \phi + \phi^{*_H})$ is a projectively flat bundle. If it is stable, the harmonic section $[H]$ of $\Met(E)$ is unique.
\end{thm}

As stated, the NAH is a correspondence between traceless Higgs bundles and projectively flat bundles. Tensoring the projectively flat bundle with a smooth line bundle with connection does not change the representation into $\mathrm{PGL}(n,\C)$, but has the effect of tensoring the Higgs bundle with a holomorphic line bundle. Thus, the NAH also gives a correspondence between representations and Higgs bundles up to tensoring with a line bundle. Note that tensoring with a line bundle changes the degree by a multiple of the rank and does not affect stability

\begin{remark} \label{rmk: polystable} 
If $(E, \dbar_E, \phi)$ is strictly polystable, i.e., it is a direct sum of $k>1$ stable Higgs bundles $(E_i, \dbar_{E_i}, \phi_i)$ of the same slope, then there is a $(k-1)$-parameter family of choices of harmonic section of $\Met(E)$. The different choices are related by the choice of relative scalings of the $k$ factors $\Met{E_i}$. In particular, the corresponding harmonic maps to the symmetric space are equal up to isometry.

The pullback to $M = \prod_{i=1}^k \Met(E_i)$ of the Killing form on $\Met(E)$ is independent of the choice of relative scaling. Note that because of the normalization, it is not the product of the Killing forms on $\Met(E_i)$. In fact, $M$ has a $k$-dimensional family of invariant metrics $\nu$, corresponding to rescaling the invariant metric on each factor. We emphasize that the definition of a minimal section of $M$ depends on the choice of $\nu$. In section \ref{GHigg} we will work with a general invariant metric $\nu$ on $M$, and so we will avoid embedding $M$ into $\Met(E)$. 
\end{remark}

\begin{defn}
We will say that a Higgs bundle $(E,\overline{\partial}_E,\phi)$ is minimal with respect to an invariant metric $\nu$ if $\nu(\phi^2)=0.$
\end{defn}
\end{subsection}

\begin{subsection}{Symmetric spaces and $G$-Higgs bundles}\label{2.4}
If the reader is only interested in Theorem B for $\mathrm{PGL}(n,\C)$, then this subsection will be skipped. Toward Theorem A for $\mathrm{PSL}(n,\R)$, it's not necessary to read this subsection too carefully: we just use the basic definitions from this subsection in discussing the Hitchin map in section \ref{sec: hitchin section}.

\begin{subsubsection}{Symmetric spaces of non-compact type}

To us, a symmetric space of non-compact type is a connected and simply connected Riemannian manifold with an inversion symmetry about each point, whose de Rham decomposition contains only non-compact symmetric spaces, and no factors of $\R$. The isometry group of a symmetric space of non-compact type is always semisimple with no compact factors and trivial center. Such groups are our primary interest, but in order to make it easy to apply the results to other standard settings, we allow certain subgroups and non-faithful actions as well. 
\begin{defn}
A Lie group $G$ is admissible if it is connected and semisimple with finite center and no compact factors. An admissible pair $(G,K)$ is an admissible Lie group $G$ together with a maximal compact subgroup $K$.
\end{defn}

\begin{defn}
An action of an admissible Lie group $G$ is essentially faithful if its kernel is discrete. 
\end{defn}

If $(G,K)$ is an admissible pair, we will always write $\g = \kk \oplus\p$ for the corresponding Lie algebra decomopsition. Let $N$ be the pointed space $G/K$.
Using a generalization of the $\g$-valued 1-form $\varphi$, one can show that $G$-invariant metrics on $N$ are in bijection with $K$-invariant positive definite bilinear forms on $\p$. Any choice of metric makes $N$ into a symmetric space of non-compact type with an essentially faithful transitive $G$-action.

\begin{defn}An admissible triple $(G,K,\nu)$ is an admissible pair $(G,K)$ together with a $K$-invariant positive definite bilinear form $\nu$ on $\p$.
\end{defn}

If $N$ is a symmetric space of non-compact type, and $h$ a point of $N$, this determines canonically a Lie algebra $\g$ of Killing fields and a decomposition $\g = \kk_h \oplus\p_h$, with $\kk_h$ the Killing fields vanishing at $h$, as well as the metric $\nu$. The pair $(G = \mathrm{Aut}_0(N), K = \mathrm{Stab}_G(h))$ is admissible with these Lie algebras, but $N$ does not uniquely determine the Lie group $G$.

\begin{remark}\label{ssform} If $G$ is simple, then $\g$ has a unique invariant bilinear form up to scale. For the purpose of studying minimal surfaces in $N$, scaling the metric is unimportant. However, if $N$ is reducible, then different scalings of its different factors determine geniunely different notions of minimal surfaces. Thus, in the reducible case $N$ contains an important additional piece of information that $G$ does not.
\end{remark}

Theorem X.2.6 of \cite{KN} proves that the formulas for the curvature and connection of $\Met(E)$ in Proposition \ref{prop: grad and R} continue to hold for $N$. Observe that these formulas only depend on $G$, and not on the choice of invariant bilinear form. Alternatively, we can deduce them by isometrically embedding $N$ as a totally geodesic subspace of a product of symmetric spaces of the type $\Met$.

\begin{prop}\label{essfaithex}
If $(G,K,\nu)$ is an admissible triple, there is a representation $\sigma = \prod_i \sigma_i: G \to \prod_i \mathrm{SL}_{n_i}(\C)$ sending $K$ to $\prod_i \mathrm{SU}(n_i)$, and constants $a_i$, such that the induced map $(G/K, \nu) \to \prod_i (\Met(\C^{n_i}), a_i \tr)$ is a totally geodesic isometry.
\end{prop}

\begin{proof}This is essentially Theorem 2.6.5 of \cite{Eb}. The $K$-invariant form $\nu$ on $\p$ extends uniquely to a $G$-invariant form on $\g$ that we also call $\nu$. It also determines a positive definite form $\nu_K$ on $\g$, which is equal to $\nu$ on $\p$ and $-\nu$ on $\kk$. Take an orthonormal basis of $\g$ with respect to the form $\nu_K$, and use the adjoint representation to map $G$ to $\mathrm{SL}({\mathrm{dim}(\g),\R)}$. Since the adjoint representation preserves the splitting of $\g$ into simple pieces $\g_i$, it actually maps to $\prod_i \mathrm{SL}(n_i, \R)$, where $n_i = \mathrm{dim}(\g_i)$. Including $\mathrm{SL}(n_i, \R)$ into $\mathrm{SL}(n_i, \C)$ gives the required representation.
\end{proof}

The most important invariant of a symmetric space $N$ of non-compact type is its rank, also called the (real) rank of $G$ or $\g$, which is the largest dimensional subspace of its tangent space at any point on which the sectional curvature vanishes. Equivalently, this is the dimension of a maximal abelian subalgebra of $\g$ contained in $\p$. We call such an algebra a maximal toral subalgebra of $\p$, using the word toral to emphasize that it necessarily consists of semisimple elements.

\subsubsection{Invariant polynomials}\label{subsec: invariant polynomials}
Let us now additionally fix a maximal toral subalgebra $\aaa$ of $\p$. The Weyl group of $\aaa$ (sometimes called the restricted Weyl group of $\aaa$) is the normalizer of $\aaa$ in $K$ modulo its centralizer. The roots of $\aaa$ (sometimes called restricted roots) are the nonzero characters of $\aaa$ that arise in its adjoint action on $\g$. We reserve the word Cartan subalgebra for a (complex) maximal toral subalgebra of the complexification $\g^\C$. The real Lie algebra $\g$ is called split if the complexification of a maximal toral subalgebra of $\p$ is a Cartan subalgebra of $\g^\C$. A Lie group $G$ is called split if its Lie algebra is split.

Let $K^\C\subset G^\C$ and $\aaa^\C\subset \p^\C\subset \g^\C$ be complexifications of $K$, $\aaa$, and $\p$ respectively. We will use the following in section 4. Let $\pK$ be the ring of $K^{\C}$-invariant complex polynomials on $\p^{\C}$, and $\aW$ the ring of $W$-invariant polynomials on $\aaa^{\C}$. The real version of the Chevalley restriction theorem says that the restriction map from $\pK \to \aW$ is an isomorphism \cite[Theorem 6.10]{Helg}, \cite[Theorem 7]{Vin}. Another theorem attributed to Chevalley asserts that $\aW$ is free on $r$ generators, where $r$ is the dimension of $\aaa$, i.e., the rank of $N$ \cite[page 54]{Hum}. Thus, $\pK$ is free on $r$ generators; for example, if $G = \mathrm{PGL}(n,\R)$, so that $\p^{\C}$ consists of traceless complex symmetric matrices and $\aaa^{\C}$ consists of traceless complex diagonal matrices, then the elementary symmetric polynomials $e_i$, $i= 2, \ldots, n$, applied to the entries of elements in $\aaa^{\C}$ determines a minimal set of generators for $\pK$.

Recall that any $X\in \p^{\C}$ has a canonical additive Jordan decomposition $X = X_s + X_n$ into commuting semisimple and nilpotent parts, with both $X_s$ and $X_n$ in $\p^{\C}$ \cite[\S 1.4]{Vin}. Since $X_s$ is semisimple it is contained in the complexification of some maximal toral subalgebra $(\aaa')^{\C}$ that is isomorphic to $\aaa^{\C}$, the isomorphism $q: (\aaa')^{\C} \to \aaa^{\C}$ being unique up to $W$ \cite[\S 3.2]{Vin}. This assignment $X\mapsto q(X_s)/W$ is dual to the Chevalley isomorphism, since if $p\in \pK,$ then $p(X) = p(q(X_s)).$

We say that an element $X \in \p^{\C}$ is in general position if the corresponding element $q(X_s)$ of $\aaa^{\C}/W$ is not in the kernel of any restricted root. This is equivalent to saying that the centralizer of $X_s$ in $\g^{\C}$ is equal to the centralizer of $(\aaa')^{\C}$ in $\g^{\C}$ \cite[\S 3.2]{Vin}. If $\g$ is split, then the restricted roots are the same as the roots of $\g^{\C}$. In this case $X$ is in general position if and only if its semisimple part is regular, which is well-known to imply that $X = X_s$. The following lemma says that this is true in the non-split case as well.

\begin{lem} If $X \in \p^{\C}$ is in general position, then it is semisimple.
\end{lem}

\begin{proof}
Applying the $K^{\C}$-adjoint-action, we can assume that the semisimple part $X_s$ lies in the complexified maximal toral subalgebra $\mathfrak{a}^{\C}$ (this is just for notational convenience). By \cite[Proposition 6.40]{Knapp} or \cite[Lemma 3.2.1]{GL}, the centralizer of $\aaa$ in $\p$ is equal to $\aaa.$ Complexifying, it follows that the centralizer of $\aaa^{\C}$ in $\p^{\C}$ is $\aaa^{\C}$. Hence, as $X$ is in general position, the centralizer of $X_s$ in $\p^{\C}$ is $\aaa^{\C}$. Since $X_n$ is in $\p^{\C}$ and centralizes $X_s$, it therefore lies in $\aaa^{\C}.$ But $\aaa^{\C}$ is the Lie algebra of an algebraic torus and hence consists entirely of semisimple elements \cite[Chapter 3, \S 8.5]{Borel}. It follows that $X_n = 0$, so $X$ is semisimple. 
\end{proof} 
The following lemmas will be used in the proof of Theorem A in section 4.
\begin{lem} \label{lem: reg semisimp} If $X \in \p^\C$ is in general position, $Y \in \p^{\C}$ is arbitrary, and every invariant polynomial $p \in \pK$ satisfies $p(X) = p(Y)$, then $X = \mathrm{Ad}_k Y$ for some $k$ in $K$.
\end{lem}
\begin{proof}
    Let $Y=Y_s+Y_n$ be the Jordan decomposition of $Y$.  According to \cite[Theorem 3]{Vin}, two points in $\mathfrak{p}^{\C}$ fail to be distinguished by invariant polynomials if and only if their semisimple parts belong to the same $K^{\C}$-adjoint-orbit. Thus, because $p(Y)=p(Y_s)$ for all invariant polynomials $p,$ $Y_s$ lies in the same $K^{\C}$-orbit as $X_s=X$. This implies that $Y_s$ is regular, and moreover the lemma above gives that $Y_n=0$. We deduce that $X$ and $Y$ live in the same $K^{\C}$-adjoint-orbit. 
\end{proof}

\begin{lem}\label{lem: same action}
    Let $Y\in \p^{\C}$ be in general position and let $k$ and $k'$ be two elements of $K$ such that $\mathrm{Ad}_k Y = \mathrm{Ad}_{k'} Y$. If $Z\in \p^{\C}$ commutes with $Y$, then $\mathrm{Ad}_k Z = \mathrm{Ad}_{k'} Z$.
\end{lem}
\begin{proof}
    This follows from Proposition 2.20.18 in \cite{Eb}, which says that if $g\in K$ and $A\in \p$ satisfy $\textrm{Ad}_{g}A=A$, then for all $E$ in the intersection of all maximal abelian subspaces of $\p$ containing $A,$ we have $\textrm{Ad}_gE=E.$ Note that there is a typo in \cite{Eb}: the word “maximal” is omitted. By breaking into real and imaginary parts, the result still holds if we take $A\in \p^{\C}.$ 
    
    In the setting of the lemma, since $Y$ is in general position, the centralizer of $Y$ in $\p^{\C}$ is a maximal abelian subspace of $\p^{\C}.$ Since the hypothesis implies that $\textrm{Ad}_{k'k^{-1}} Y=Y$, we have that for all $Z\in \p^{\C}$ commuting with $Y,$ $\textrm{Ad}_{k'k^{-1}} Z=Z$, or $\mathrm{Ad}_k Z = \mathrm{Ad}_{k'} Z$. 
\end{proof}

\end{subsubsection}

\begin{subsubsection}{$G$-Higgs bundles}\label{subsubsec: G higgs}

Let $N$ be a symmetric space of non-compact type, and let $G = \mathrm{Isom}(N)$. Let $M$ be a flat Riemannian bundle over a closed surface $S$, each of whose fibers is isometric to $N$. If $\tilde{S}$ is a universal cover of $S$, then $M$ defines, up to conjugacy, a holonomy representation from the Deck group of $\tilde{S}$ to $G$. Similar to the $\textrm{PGL}(n,\C)$ case, the bundle is reductive if the Zariski closure of its holonomy representation is a reductive subgroup of $G$. The bundle is irreducible if its holonomy is not contained in a non-trivial parabolic subgroup of $G$. 
Corlette's existence result on harmonic maps is formulated in this setting \cite{C}.
\begin{thm}[$G$-NAH I] There exists a harmonic section $h:S\to M$ if and only if the holonomy of $M$ is reductive. If the holonomy is irreducible, then $h$ is unique.
\end{thm}

Let $Q$ be the bundle whose fiber at $z \in S$ is the $G$-torsor of isometries from $N$ to $M_z$. Then $Q$ is a $G$-principle bundle that inherits a flat connection $A$ (an equivariant $\g$-valued $1$-form satisfying some properties \cite[Chapter 2.1]{DK}) from the flat connection on $M$. If we fix once and for all a point $h_0$ in $N$ with stabilizer $K$ and corresponding decomposition $\g = \kk \oplus \p$, then a section $h$ of $M$ gives a reduction $Q^h$ of the structure group of $Q$ to $K$, whose fiber at $z$ is the subset of isometries sending $h_0$ to $h(z)$. The pullback of the vertical tangent bundle of $M$ by the section $h$ is therefore identified with $Q^h \times_K \p$. Write $\nabla^h$ for the pullback of the Levi-Civita connection by $h$ to $Q^h \times_K \p$, given by Proposition \ref{prop: grad and R}. The harmonicity of $h$ is equivalent to the holomorphicity of $\partial h$ with respect to the $(0,1)$-part of $\nabla^h$. Similarly, the conformality of $h$ with respect to the metric $\nu$ on $N$ is equivalent to the condition $\nu((\partial h)^2) := \nu(\partial h \otimes \partial h)= 0$. 

\begin{defn} We say that $(M,h)$ is a $G$-harmonic bundle if $h$ is harmonic. It is a minimal $G$-harmonic bundle with respect to a metric $\nu$ if $h$ is minimal with respect to $\nu$. 
\end{defn}

This leads us to a definition of $G$-Higgs bundles. Let $(G,K)$ be any admissible pair. Given a principle $K^{\C}$-bundle $P$ over $S,$ we write  $\mathrm{ad}\p^\C = P \times_{K^\C} \p^\C$. A holomorphic principle $K^{\C}$-bundle is equivalently a pair $(P,A^{0,1})$ consisting of a smooth principal $K^\C$-bundle on $S$ and a $(0,1)$-connection $A^{0,1}$.

\begin{defn}
A $G$-Higgs bundle is a pair $(P,A^{0,1},\phi),$ where $(P,A^{0,1})$ is a holomorphic principal $K^\C$-bundle on $S$ and $\phi$ is a holomorphic section of $\textrm{ad}\p^\C\otimes \mathcal{K}$ called the Higgs field. 
\end{defn}

The other direction of NAH is also true in this context. A reduction $P^c$ of the structure group of $P$ to $K \subset K^\C$ defines an adjoint operator $*_c$ on $\mathrm{ad}\g^\C=P\times_{K^{\C}}\g^{\C}$. The Chern connection of $A^{0,1}$ with respect to this reduction is $A^{0,1} + (A^{0,1})^{*_c}$, and its curvature $F^c$ is a real $\mathrm{ad}\kk^\C$-valued $2$-form. The connection $\theta = A^{0,1} + (A^{0,1})^{*_c} + \phi +\phi^{*_c}$ is flat if and only if the curvature $F^c$ satisfies
$$
F^c + [\phi,\phi^{*_c}]=0.
$$
Stability and polystability for $G$-Higgs bundles are defined analogously to the $\textrm{PGL}(n,\C)$ case (see \cite{G-P}). The analog of Proposition \ref{scale} remains true for $G$-Higgs bundles. The following is proved in \cite[Theorem 2.24]{GPGR}.

\begin{thm}[$G$-NAH II] 
Let $(P,A^{0,1},\phi)$ be a polystable $G$-Higgs bundle. There exists a reduction of the structure group $P^c$ such that $(P^c \times_K G,A^{0,1} + (A^{0,1})^{*_c} + \phi +\phi^{*_c})$ is a flat $G$-principal bundle (the reduction $P^c\subset P$ therefore determines a harmonic section of the associated flat Riemannian bundle with fibers isometric to $G/K$). If $(P,A^{0,1},\phi)$ is stable, the reduction is unique.
\end{thm}

\begin{defn}
A $G$-Higgs bundle $(P,A^{0,1},\phi)$ is called minimal with respect to an invariant bilinear form $\nu$ on the symmetric space $G/K$ if $\nu(\phi^2) = 0$.
\end{defn}

Let's give a word on the relation to ordinary Higgs bundles. First, suppose that $H$ is a complex Lie group. Since we only define $G$-Higgs bundles for real Lie groups $G$, to make sense of an $H$-Higgs bundle, we must first restrict scalars from $\C$ to $\R$. The complexification of $G = \mathrm{Res}_{\C/\R}H$ is isomorphic to $H \times \overline{H}$. With repsect to any Cartan involution, both $\kk^\C$ and $\p^\C$ are isomorphic to the Lie algebra $\h$. In this way, we recover the usual notion of an $H$-Higgs bundle for a complex group $H$. 

Next, if $(P,A^{0,1},\phi)$ is a $\mathrm{PGL}(n,\C)$-Higgs bundle, the obstruction to holomorphically extending the structure group of $P$ to $\mathrm{GL}(n,\C)$ is a class in $H^2(S, \mathcal{O}^*)$, where $\mathcal{O}^*$ is the sheaf of invertible holomorphic functions. This space is trivial. We have to choose an extension, but having done so, we get a classical Higgs bundle with $E$ the associated vector bundle. A reduction of the structure group from $\mathrm{PGL}(n,\C)$ to $\mathrm{PU}(n)$ determines a Hermitian metric on $E$ up to scaling by a real-valued function on $S$.

Finally, if $G$ and $G'$ are two real semisimple groups, a homomorphism $\sigma: G \to G'$ turns a $G$-Higgs bundles into a $G'$-Higgs bundle. If $(P,A^{0,1},\phi)$ is stable, then $(\sigma(P), \sigma(\phi))$ will be polystable, but may not be stable (see \cite{BGG}). 

Putting together the previous three paragraphs, this gives a way of associating a classical Higgs bundle to a $G$-Higgs bundle provided we are given a representation of $G$ into $\mathrm{PGL}(n,\C)$. We will use this association in order to apply the asymptotic decoupling estimates to $G$-Higgs bundles, under the assumption that the representation is essentially faithful. 

\end{subsubsection}

\end{subsection}

\end{section}

\begin{section}{Convergence of the second variation along the $\mathbb{R}^+$-flow}
The purpose of this section is to prove Theorem B.
\begin{subsection}{The second variation of area}
Before studying the second variation of area for maps to symmetric spaces, we give a general formula for the second variation of area for a minimal immersion to a Riemannian manifold. Let $S$ be a Riemann surface and $h : S\times (-\epsilon, \epsilon) \to N$ a one-parameter variation of a minimal map with pullback bundle $h^* TN$. We write $h$ for the initial map and hope there is no confusion. Let $\nabla$ be the connection on $h^*TN$ obtained by pulling back the Levi-Civita connection from  $N$. We define sections $\dot{h}$ and $\ddot{h}$ of $h^*TN|_{S\times\{0\}}$ by 
$$
\dot{h}=dh\Big ( \frac{\partial}{\partial t}\Big )|_{t=0}, \hspace{1mm} \ddot{h}=\Big(\nabla_{\frac{\partial}{\partial t}} dh\Big ( \frac{\partial}{\partial t}\Big )\Big )|_{t=0}.
$$
Similarly, we have $1$-forms with values in  $h^*TN\otimes_{\R}\C$, $$
\partial \dot{h}= \Big (\nabla^{1,0} dh\Big ( \frac{\partial}{\partial t}\Big ) \Big)|_{t=0}, \hspace{1mm} \overline{\partial} \dot{h}= \Big (\nabla^{0,1} dh\Big ( \frac{\partial}{\partial t}\Big ) \Big)|_{t=0}.$$
Let $a(t)$ be the area form induced on $S$ from the map $h$ at time $t$, and let $\ddot{a} = \frac{d^2}{dt^2} \big|_{t= 0} a(t)$. In this context, the standard second variation of area formula takes the following form. We remark that we make no assumption about $\dot{h}$ being normal, i.e., that $\langle \dot{h},dh\rangle =0.$

\begin{prop}[Second variation of area] If $h$ is a minimal immersion at time 0, then the second derivative of the induced area form at time 0 is given by 
\begin{equation} \label{eqn: second var}
\frac{1}{2} \ddot{a} =  \langle \partial \dot{h}, \overline{\partial} \dot{h} \rangle + \langle R(\dot{h}, \partial h) \dot h,\overline{\partial} h \rangle - \frac{2 |\langle \partial \dot{h}, \partial h \rangle|^2}{\langle \partial h, \overline{\partial} h \rangle} + \frac{1}{2}d(\langle \ddot{h},\overline{\partial} h\rangle - \langle \partial h, \ddot{h}\rangle).
\end{equation}
\end{prop}

\begin{proof} From equation (\ref{adens}) we have $a(t) = \sqrt{\langle \partial h_t, \overline{\partial} h_t \rangle ^2 - |\langle \partial h_t, \partial h_t \rangle |^2}$. Set $f(t)=\langle \partial h_t, \overline{\partial} h_t \rangle$ and $g(t)=\langle \partial h_t, \partial h_t \rangle$, the latter of which satisfies $g(0)=0$ because $h$ is conformal. We Taylor expand $a(t)$ in $t$ and use $g(0) = 0$ to find that 
\begin{equation}\label{eqn: Taylor area}
    \ddot{a}= \ddot{f}- \frac{|\dot{g}|^2}{f(0)}.
\end{equation}
Toward computing the right hand side of (\ref{eqn: Taylor area}), we record that, since the Levi-Civita connection is torsion-free,
\begin{align*}
\Big (\nabla_{\frac{\partial}{\partial t}} \partial h_t\Big )|_{t=0} &= \partial \dot{h} \\
\Big (\nabla_{\frac{\partial}{\partial t}} \nabla_{\frac{\partial}{\partial t}} \partial h_t\Big )|_{t=0} &= R(\dot{h}, \partial h) \dot h + \partial \ddot{h},
\end{align*}
and similarly for $\overline{\partial} h_t$. We compute each term of (\ref{eqn: Taylor area}) separately. The second term is handled easily: $$\frac{|\dot{g}|^2}{f(0)}=\frac{4 |\langle \partial \dot{h}, \partial h \rangle|^2}{\langle \partial h, \overline{\partial} h \rangle}.$$ For the first term,
\begin{align*}
    \ddot{f} &= 2 \langle \nabla_{\frac{\partial}{\partial t}} \partial h_t, \nabla_{\frac{\partial}{\partial t}} \dbar h_t \rangle + \langle \nabla_{\frac{\partial}{\partial t}}\nabla_{\frac{\partial}{\partial t}}\partial h_t, \dbar h_t \rangle + \langle \partial h_t, \nabla_{\frac{\partial}{\partial t}}\nabla_{\frac{\partial}{\partial t}} \dbar h_t \rangle \bigg|_{t=0}\\
    &= 2\langle \partial \dot{h}, \overline{\partial} \dot{h} \rangle + \langle R(\dot{h},\partial h)\dot{h},\overline{\partial}h\rangle + \langle \partial h, R(\dot{h},\overline{\partial} h)\dot{h}\rangle +\langle \partial\ddot{h},\overline{\partial} h\rangle + \langle \partial h,\overline{\partial} \ddot{h}\rangle.
\end{align*}
By the symmetries of the Riemannian curvature tensor, the two curvature terms are equal. For the rightmost two terms, we have
\begin{align*}
    \langle \partial\ddot{h},\overline{\partial} h\rangle + \langle \partial h,\overline{\partial} \ddot{h}\rangle &= \partial\langle \ddot{h},\overline{\partial} h\rangle - \overline{\partial}\langle \partial h, \ddot{h}\rangle \\
    &= d(\langle \ddot{h},\overline{\partial} h\rangle - \langle \partial h, \ddot{h}\rangle).
\end{align*}
Putting this all together yields the proposition. 
\end{proof}

If $h$ is a minimal section of a flat Riemannian bundle $M$ over $S$ that is an immersion, then, as in section \ref{sec: flat Riemannian}, we replace the bundle $h^*TN$ by the vertical tangent bundle $h^*T^\mathrm{vert}M$, and the Levi-Civita connection on each fiber by its natural extension to the whole bundle $T^\mathrm{vert}M$. Then equation \ref{eqn: second var} still holds, as one checks in local flat charts.

If $h: S \to M$ is a minimal section that is not necessarily an immersion, define the space $\mathrm{Var}(h)$ of infinitesimal variations of $h$ to be the space of smooth sections of $h^*T^\mathrm{vert}M$ with compact support on $S$. The stability form $Q_h$ of $h$ is the quadratic form on $\mathrm{Var}(h)$ defined for $X \in \mathrm{Var}(h)$ by
\begin{equation}\label{stabform}
    Q_h(X) = 2\int_S \langle \partial X, \overline{\partial} X \rangle - \frac{2 |\langle \partial X, \partial h \rangle|^2}{\langle \partial h, \overline{\partial} h \rangle} + \langle R(X, \partial h) X,\overline{\partial} h \rangle.
\end{equation}

\begin{prop} \label{prop: second variation integral} If $h$ is a minimal section and $h(t)$ is a one-parameter variation with $\dot{h} = X \in \mathrm{Var}(h)$, and $A(t)$ the total area of a compact subsurface containing the support of $X$ and of $\ddot{h}$, then
\[
\frac{d^2}{dt^2}\bigg|_{t=0} A(t) = Q_h(X).
\]
\end{prop}

\begin{proof}
The area form is non-degenerate on the open set $U\subset S$ on which $h$ is an immersion, and so we can apply the formula (\ref{eqn: second var}) on $U$. But it is easily checked that both the area form and the right hand side of the formula (\ref{eqn: second var}) smoothly extend to all of $S$. The result follows from dominated convergence. By the divergence theorem, the final term in formula \eqref{eqn: second var} vanishes upon integration.
\end{proof}

\begin{defn} We call $X$ a destabilizing variation of $h$ if $Q_h(X) < 0$. $h$ is called unstable if it has a destabilizing variation. The index $\mathrm{Ind}(h)$ of $h$ is the maximal dimension of a linear subspace of destabilizing variations.
\end{defn}

Henceforth we refer to infinitesimal variations simply as variations. We also note that the index may be infinite, but if $S$ is closed then by elliptic theory the index is finite. 

The log cutoff trick (see \cite[section 4.4]{MSS} for details and explanation) implies the following proposition, which is essential to the proof of Theorem B.

\begin{prop}\label{prop: log cutoff}
If $U = S-B$, for $B$ a discrete set, then $\mathrm{Ind}(h|_U) = \mathrm{Ind}(h)$.
\end{prop}
Consequently, if $S$ is closed then $\textrm{Ind}(h|_U)$ is finite. Another elementary proposition of a similar flavor is:

\begin{prop} \label{prop: tot geodesic index} If $M$ has fiberwise nonpositive sectional curvature, $\iota: M' \to M$ is the inclusion of a subbundle that is fiberwise totally geodesic, and $h = \iota \circ h'$, then $\mathrm{Ind}(h) = \mathrm{Ind}(h')$.
\end{prop}

\subsubsection{Minimal harmonic bundles} We record the stability form for a minimal harmonic bundle $(E,D,H)$. Recall the $\textrm{End}^0(E)$-valued $1$-form $\varphi$ on $\mathbb{P}\textrm{Met}(E),$ defined fiberwise by equation (\ref{eqn: varphi}) and which, as in section 2.3, we interpret as an isomorphism from $T\mathbb{P}\textrm{Met}(E)$ to $\textrm{End}^0(E)$. We use $\varphi$ to identify $\mathrm{Var}(H)$ with the space of smooth sections of $\mathrm{End}^0_H(E)$. One way of thinking about this identification is that if $X$ is a smooth section of $\mathrm{End}^0_H(E)$, then the one-parameter variation of metrics $$H^t = e^{-t\overline{X}^\vee}H e^{-tX}$$ 
has derivative $X$ at time zero.

If $X$ is a variation, then $\partial X$ should take values in the complexified tangent bundle, so it becomes an $\mathrm{End}^0(E)$-valued $1$-form. Recall that under this isomorphism, the (complexified) Riemannian metric becomes $\langle X, Y\rangle = 2n \tr(XY)$. We write the Hermitian norm corresponding to $H$ as $|X|^2_H = \langle X, X^{*_H} \rangle$. The operator $\partial$ on this bundle is the $(1,0)$-part of $\nabla^H$ defined in section \ref{2.3}.

\begin{prop}\label{stabformharm}
 Let $(E,D,H)$ be a minimal harmonic bundle on a closed Riemann surface $S$ with associated Higgs field $\phi$. If $X$ is a smooth section of $\mathrm{End}_H^0(E)$, then the stability form of the section $[H]$ of the associated flat Riemannian bundle $\Met(E)$ is $$Q_H(X) = 2\int_{S} |\partial X|_H^2-\frac{2|\langle \partial X,\phi\rangle|^2}{|\phi|_H^2}+|[X,\phi]|_H^2.$$
\end{prop}
\begin{proof}
By definition, $\phi = \varphi(\partial H)$, and since the identification $\varphi$ is implicit, we should replace ``$\partial h$'' in Equation \eqref{stabform} with $\phi$. Because $X = X^{*_H}$, the curvature term becomes
\[
\langle R(X,\phi)X,\phi^{*_H} \rangle = 2n\tr(-[[X,\phi],X] \phi^{*_H}) = 2n\tr([X,\phi][\phi^{*_H},X]) = |[X,\phi]|^2_H.
\]

\end{proof}

\end{subsection}

\begin{subsection}{The limiting objects}\label{sec: 3.2}
There is a $\mathbb{C}^*$-action on the moduli space of Higgs bundles on a closed Riemann surface $S$, 
$$
\gamma\cdot (E,\overline{\partial}_E,\phi) = (E,\overline{\partial}_E,\gamma\phi).
$$ By Proposition \ref{scale}, the action restricts to the space of (poly)stable Higgs bundles. In this paper, we are only interested in the restricted $\mathbb{R}^+$-action. Since stable Higgs bundles have unique harmonic metrics, the non-abelian Hodge correspondence transfers the $\mathbb{R}^+$-action to an $\mathbb{R}^+$-action on the moduli space of projectively flat bundles with irreducible holonomy. Given a stable Higgs bundle $(E,\overline{\partial}_E,\phi)$ on a closed Riemann surface $S$, by taking $R\mapsto (E,\overline{\partial}_E,R\phi),$ $R>0,$ we obtain a family of harmonic bundles parametrized by $\R^+$ that we call a Hitchin ray.

In this section, we define the apartment bundle associated to a generically semisimple traceless Higgs bundle on a connected Riemann surface $S$ (not necessarily closed). It is a flat Riemannian bundle over $S$ with real affine fibers together with a canonical section that serves as a model of the limit as $R \to \infty$ of any family of harmonic bundles along a Hitchin ray. The apartment bundle is equivalent to the data of the cameral cover and an equivariant harmonic map from the universal cover of a component of the cameral cover (see below).

Let $S$ be a Riemann surface and let $(E, \dbar_E, \phi)$ be a traceless Higgs bundle on $S$.

\begin{defn}
    The critical set $B \subset S$ of $(E, \dbar_E, \phi)$ is the subset of points $p \in S$ on which the number of generalized eigenspaces of $\phi$ at $p$ is strictly less than its maximum value on $S$.
\end{defn}

By holomorphicity of $\phi$, $B$ is a discrete subset of $S$. Let $m \leq n$ be such that at each point $p \in S - B$, $\phi$ has exactly $m$ generalized eigenspaces $E_i$ at $p$. Generically regular semisimple Higgs bundles are studied in \cite{Mo}. The Higgs field $\phi$ is generically regular semisimple if and only if $m = n$. We slightly relax the condition.

\begin{defn} $\phi$ is generically semisimple if it is semisimple away from a discrete subset of $S$.
\end{defn}
Henceforth we assume that $(E,\overline{\partial}_E,\phi)$ is generically semisimple, generically with $m\leq n$ eigenspaces. A lot of the constructions below go through in the general case using the Jordan decomposition, but for this paper it suffices to consider only generically semisimple Higgs bundles.

For the following constructions, we work entirely on $S - B$. For each $p \in S - B$, there is a set $\{E_{p,i}, i \in \mathrm{Spec}_p\}$ of generalized eigenspaces for $\phi$ at $p$. The notation $\mathrm{Spec}_p$ is chosen because this set is equal to the fiber of the (reduced) spectral curve at $p$. The reduced spectral curve is a covering space over $S-B$ (see \cite[section 5.4]{DonSpec}), so the splitting at an arbitrary point $p \in S-B$ can be extended to a disk in $S-B$ around $p$. That is, in such a disk, we have a holomorphic splitting into subbundles $E=\oplus_{i\in \mathrm{Spec}_p}E_i$, where the fiber of $E_i$ at $p$ is $E_{p,i}$, with holomorphic projections $\pi_i \in \mathrm{End}(E)$ from $E$ to $E_i$, and $\phi$ has eigen-$1$-forms $\phi_i.$  Around $p$ we can further write $\phi = \sum_{i\in \mathrm{Spec}_p} \phi \circ \pi_i$, and the Jordan decomposition of $\phi$ splits over the sum $\phi = \sum_{i\in \mathrm{Spec}_p} \phi \circ \pi_i$. We observe that $\phi$ is semisimple on all of $S-B,$ since if any $\phi\circ \pi_i$ has a non-zero nilpotent part, then this would persist over a neighbourhood of $p,$ since $\phi\circ\pi_i$ not being a multiple of the identity is an open condition (see also Lemma \ref{lem: phi = phis}). Hence if $\phi$ is generically semisimple, then $\phi = \sum_{i \in \mathrm{Spec}_p} \phi_i \pi_i$ on $S-B$. 
\begin{defn}
    The extended toral bundle $F_\phi$ of $(E,\overline{\partial}_E,\phi)$ is the $m$ dimensional sub-bundle of $\mathrm{End}(E)|_{S-B}$ spanned at each point $p$ by the projections $\{\pi_i, i \in \mathrm{Spec}_p\}$. The toral bundle is the $m-1$ dimensional sub-bundle $F^0_\phi = F_\phi \cap \mathrm{End}^0(E)$.
\end{defn}
 Equivalently, the fiber of $F^0_\phi$ over each point $p$ of $S-B$ is the center of the centralizer of $\phi(p)$ in $\mathrm{End}^0(E)$. If $\phi$ is regular semisimple, this centralizer is already abelian. In this case, the toral bundle (resp. extended toral bundle) is the bundle of regular centralizers denoted $\mathbf{c}_X$ in \cite{DonGai}.

Fix a generically semisimple traceless Higgs bundle $(E,\overline{\partial}_E,\phi)$ on a Riemann surface $S$, and let $F_\phi$ be its extended toral bundle. The transition maps of $F_\phi$ only permute the projections $\pi_i$ (and moreover can only permute two projections $\pi_i$ of the same rank), and hence it is contained in a (finite) product of symmetric groups. Therefore, to define any object on $F_\phi$ (flat connection, real structure, etc.), it suffices to define it on the $\pi_i$'s and check invariance under the permutation action of the monodromy. In this way, we see that $F_\phi$ has a natural real structure and flat connection defined by the condition that the $\pi_i$'s are real and flat. Concretely, if $X$ is a section of $F_\phi$, let $dX = \partial X + \dbar X$ be its derivative with respect to the flat connection, and let $X^\dagger$ be its conjugate with respect to the real structure. We define $d$ and $\dagger$ by $d\pi_i=0$ $\pi_i^\dagger=\pi_i,$ and hence if $X = \sum_i X_i \pi_i$, 
\[
X^\dagger = \sum_i \overline{X_i}\pi_i
\]
and
\[
dX = \sum_i dX_i \pi_i.
\]
As well, the metric $\langle X, Y \rangle = 2n \tr(XY)$ on $\mathrm{End}(E)$ restricts to a flat metric on $F_\phi$. Via $\langle\cdot,\cdot\rangle$ and $\dagger$, we have a Hermitian metric on $F_\phi$, $|X|^2 = 2n\tr(XX^\dagger)$. If $r_i$ is the rank of $E_i$, then
\begin{equation}\label{pii}
    \langle \pi_i, \pi_j^\dagger\rangle = 2n\tr(\pi_i\pi_j^\dagger) = 2n r_i\delta_{ij}.
\end{equation}
The toral bundle $F_\phi^0$ of $\phi$ inherits the flat, metric, and real structures from $F_\phi$. Let $\R F_\phi^0$ be its real subbundle, i.e., the traceless $\R$-span of the $\pi_i$'s in $F_\phi$. Since the $\pi_i$'s are holomorphic, the flat connection is compatible with the holomorphic structure on $\textrm{End}(E)$, and hence $\phi$ can be viewed as a holomorphic $F_\phi^0$-valued $1$-form. Let $\psi = \phi + \phi^\dagger$. Since $\R F_\phi^0$ is flat and $\psi$ is a closed $\R F_\phi^0$-valued $1$-form, it defines an affine bundle $M_\phi$ whose vertical tangent bundle is $\R F_\phi^0$ together with a section $f$ of $M_\phi$. The construction is as follows. Let $(U_\alpha)$ be an open covering of $S$ on which every $\R F_\phi^0$ has a flat trivialization as $\R F_\phi^0|_{U_\alpha} =U_\alpha \times \R^{n-1}$, and let $w_{\alpha\beta} \in \mathrm{SL}(n-1,\R)$ be the locally constant transition functions. Let $(\psi_\alpha)$ be $\R^{n-1}$-valued one-forms representing $\psi$, with $\psi_\alpha = w_{\alpha\beta}\psi_\beta$. On each $U_\alpha$, choose $\R^{n-1}$-valued functions $f_\alpha$ with $df_\alpha = \psi_\alpha$. Since $dw_{\alpha\beta} = 0$, the $\R^{n-1}$-valued functions $t_{\alpha\beta} := f_\alpha - w_{\alpha\beta}f_\beta$ are locally constant. Rearranging this gives
\[
f_\alpha = w_{\alpha\beta}f_\beta + t_{\alpha\beta}.
\]
The affine transformations $x\mapsto w_{\alpha\beta}x + t_{\alpha\beta}$ define the bundle $M_\phi$, and $(f_\alpha)$ defines a section $f$ such that $df = \psi$, and therefore $\partial f = \phi$.

\begin{defn}
We call $M_\phi$ the apartment bundle associated to $(E, \dbar_E, \phi)$. 
\end{defn}

By equation (\ref{pii}), the natural metric on  $\R F_\phi^0$ is parallel with respect to the flat connection, and therefore the apartment bundle $M_\phi$ is a flat Riemannian bundle. The section $f$ is harmonic by construction. Assume that $\tr(\phi^2)=0$, so that the Higgs bundle $(E, \dbar_E, \phi)$ is minimal. Then the section $f$ is also minimal. 
\begin{defn}
The pair $(M_\phi, f)$ is the limiting object associated to the minimal generically semisimple Higgs bundle $(E, \dbar_E, \phi)$. Its index is by definition the index of $f$. 
\end{defn}

A variation $X \in \mathrm{Var}(f)$ is a smooth section of $\R F_\phi^0$ with compact support on $S-B$. Since the curvature of the fibers of $M_\phi$ vanishes, the stability form of $f$ is given by
\begin{equation} \label{eqn:ddota in F}
Q_f(X) = 2\int_S |\partial X|^2 - \frac{2 |\langle \partial X, \phi\rangle|^2}{|\phi|^2}.
\end{equation}

\subsubsection{Cameral covers} We now give an equivalent description of the limiting object, which is a slight modification of the equivariant map associated to the section $f$ as in section \ref{1.4}.

Let $\mathrm{Cam}^*$ be the covering space of $S-B$ whose fiber at a point $p$ is the set of maps $q$ from the set $\{1, \ldots, n\}$ to the set $\mathrm{Spec}_p$ of generalized eigenspaces of $\phi$ at $p$ such that the size of $q^{-1}(E_i)$ is the dimension of $E_i$. Let $\mathrm{Cam}$ be the Riemann surface with proper map $\mathrm{Cam} \to S$ whose restriction to $S - B$ is $\mathrm{Cam}^*$. Then $\mathrm{Cam}$ is the normalization of the reduction of the cameral cover as defined in \cite[section 2]{Don}.

By construction $\mathrm{Cam}^*$ has a fiberwise transitive action of the symmetric group $W_n$, which extends to an action on $\mathrm{Cam}$. Consequently, each component of $\mathrm{Cam}$ is isomorphic, and each component of $\mathrm{Cam}^*$ is a Galois covering space of $S-B$ with Deck group equal to its setwise stabilizer in $W_n$ modulo its pointwise stabilizer. 

\begin{defn}\label{def: smallcameral}
    A small cameral cover associated to $(E, \dbar_E, \phi)$ is a connected component of $\mathrm{Cam}$.
\end{defn}

Every small cameral cover is isomorphic and can be described as follows. A point $q: \{1, \dots, n\} \to \mathrm{Spec}_p$ in $\mathrm{Cam}^*$ determines a set partition of $\{1,\ldots,n\}$ by considering its fibers. The pointwise stabilizer of the component of $q$ is the same as the stabilizer of $q$ in $W_n$, which is the product of permutation groups permuting the elements of each fiber of $q$. Let $W_{(q)}$ be the subquotient of $W_n$ that acts on the set of fibers of $q$ in $\{1, \dots, n\}$, permuting those fibers of equal size. The Deck group of the small cameral cover is then a subgroup $\Gamma \subset W_{(q)}$.

Fix a small cameral cover $C$ with map $\tau:C\to S$ and restriction $C^*$ to $S-B$ with Deck group $\Gamma$. There is a canonical holomorphic $\C^n$-valued $1$-form $\phi'$ on $C^*$ whose $i$th coefficient at a point $q \in C^*$ is the eigenvalue $\phi_{q(i)}$ of $\phi$ on the eigenspace $E_{q(i)}$. Since $C^*$ is connected, each point $q \in C^*$ determines the same set partition of $\{1, \ldots, n\}$, and hence the same stabilizer and the same $W_{(q)}$. Let $V$ be the subspace of $\C^n$ invariant by the stabilizer in $W_n$ of any point $q$ in $C^*$, of some dimension $m\leq n$. Equivalently, $V$ consists of those vectors whose $i$ and $j$ components are equal whenever $q(i) = q(j)$. By construction $\phi'$ takes values in $V$. Since the stabilizer of $q$ acts trivially on $F'_\phi$, the group $W_{(q)}$ and in particular the Deck group $\Gamma$ of $C^*$ act on $V$, and one can form the associated bundle $F_V = C^* \times_\Gamma V$. It clearly inherits a flat connection, and if $\C^n$ is given its standard real structure and Hermitian metric, then these restrict to $V$ and descend to $F_V$. The bundle $F_V$ is isomorphic as a flat metric vector bundle with real structure to the extended toral bundle, and $\phi'$ is a holomorphic section. If $V^0$ is the intersection of $V$ with the subspace of $\C^n$ on which the coordinates sum to zero, then the toral bundle is isomorphic as a flat real metric vector bundle to $F^0_\phi = C^* \times_{\Gamma} V^0$, and $\phi'$ is still a holomorphic section. 

We obtain an alternative picture by extending $\phi'$ to $C$. Since the eigenvalues of $\phi$ are locally bounded on $S$, the $1$-form $\phi'$ extends canonically to $C$, and the extended $1$-form, which we still denote $\phi'$, is still valued in $V^0$, and is still equivariant with respect to the $\Gamma$ actions on $C$ and on $V^0$. Let $\psi' = \phi' + \overline{\phi}'$, which is a $\Gamma$-equivariant harmonic $\R V^0$-valued $1$-form on $C$, where $\R V^0$ is the real subbundle of $V^0$. It determines a $\R V^0$-valued cohomology class, of which it is the unique harmonic representative. Choosing a basepoint $z_0$ on $C$, the cohomology class defines a representation from $\pi_1(C)$ to $(\R V^0,+)$. Lifting to a harmonic $1$-form $\tilde{\psi'}$ on the universal cover $\tilde{C}$ and integrating from a lift of the basepoint gives a harmonic function $$\tilde{f}(z) = \int_{z_0}^z \tilde{\psi'}$$
that is equivariant by this representation and has image contained in $\R V^0$. 

One can make a natural choice of real othonormal basis for $V$ by choosing an ordering of the fibers of $q$, and doing so gives a representation $\chi:\pi_1(C)\to\R^m$ and a $\chi$-equivariant harmonic function $\tilde{f}$ from $\tilde{C}$ to $\mathbb{R}^m$ whose image is similarly contained in a hyperplane. The group $\Gamma$ acts as a reflection group on this hyperplane, and when we act on $\tilde{C}$ by a lift of a Deck transformation of $C$, say $\gamma\in \Gamma$, $\tilde{f}$ transforms according to the action of $\gamma$ on the hyperplane.

We now assume that $\tilde{f}$ is minimal. Define the equivariant area of $\tilde{f}$ to be the integral of its area form over a fundamental domain of the action of $\pi_1(C)$ on $\tilde{C}$ divided by $\textrm{deg}(\tau)$. Note that this is the area of the section $f$ of $M_\phi$. Since the representation $\chi$ acts by translations, a variation of the map $\tilde{f}$ that preserves $\chi$-equivariance is the same thing as a $\chi$-invariant map from $\tilde{C}$ to $\R^m$. Of course, this is the same thing as a map from $C$ to $\R^m$. The stability form for the equivariant area is 
$$
Q_{\tilde{f}}(X) = 2\int_C |\partial X|^2 - \frac{2|\langle \partial X, \partial \tilde{f}\rangle|^2}{|\partial f|^2}.
$$
Call a variation $X$ $\tau$-invariant if it is the pullback of a section of $F^0_\R$ on $S$. Define the $(\tau,\chi)$-invariant index of $\tilde{f}$ to be the index of the stability form restricted to $\tau$-invariant variations through $\chi$-equivariant maps.

The log cutoff trick (Proposition \ref{prop: log cutoff}) shows that we need only consider variations to be supported on $C - \tau^{-1}(B)$. But a $\tau$-invariant variation of $\tilde{f}$ supported on $C - \tau^{-1}(B)$ is the same as a variation of the section $f$, and the resulting area forms and stability forms are the same. We conclude

\begin{prop} The index of $f$ is the same as the $(\tau,\chi)$-invariant index of $\tilde{f}$.
\end{prop}
Finally, it's helpful to observe that we can also describe our limiting objects in terms of abelian Higgs bundles on $C^* = C|_{S-B}$. Fix a natural isometry $V_\R \cong \R^m$. We can think of $\R^m$ as the group $(\textrm{GL}(1,\R)^+)^m$ and also as a Riemannian symmetric space on which $(\textrm{GL}(1,\R)^+)^m$ acts. The $m$ abelian differentials $\phi'_1, \ldots, \phi'_m$ on $C$ determine a rank $m$ Higgs bundle
\[
(W, \dbar_W, \Theta) = (C \times \C^m, \dbar, \left(\begin{array}{ccc}
    \phi'_1 & &\\
     & \ddots & \\
     & & \phi'_m
\end{array}\right) ).
\]
In fact, $(W, \dbar_W, \Theta)$ is a polystable $(\textrm{GL}(1,\R)^+)^m$-Higgs bundle. See \cite{GPGR} for the theory of Higgs bundles for reductive Lie groups such as $(\textrm{GL}(1,\R)^+)^m$. We have the following.
\begin{itemize}
    \item The maximal toral bundle $F_\phi\to S-B$ is the quotient of $W|_{C^*}$ by the Deck group $\Gamma$.
    \item The Chern connection of the standard Hermitian metric on $W$ is flat, and descends to the flat connection on $F_\phi$. The real structure on $F_\phi$ comes from the standard orthogonal structure on $\C^m$. The flat metric on $F_\phi$ comes from a choice of $(\textrm{GL}(1,\R)^+)^m$-invariant metric on the symmetric space, so the scale on each factor may be chosen independently.
    \item The pullback via $\tau$ of the generically semisimple Higgs bundle $(E|_{S-B},\overline{\partial}_E,\phi)$ on $S-B$ splits as a Higgs bundle as 
    \[
(E_1 \oplus \cdots \oplus E_m,\oplus_i \overline{\partial}_{E_i}, \left(\begin{array}{ccc}
    \phi'_1 \mathrm{Id}_{E_1}& &\\
     & \ddots & \\
     & & \phi'_m \mathrm{Id}_{E_m}
\end{array}\right) ).
    \]
    \item Every $\phi_i'$ is bounded on $C^*$ and hence extends to a holomorphic $1$-form on $C$.
    
    \item The standard Hermitian metric solves the Hitchin equations for $(W, \dbar_W, \Theta)$, and the corresponding flat connection is $\oplus_i d+\phi_i'+\overline{\phi}_i'.$ Moreover, when the trivial metric is interpreted as an equivariant map $\tilde{f}:\tilde{C}\to \R^m$ from the universal cover of $C$ to the symmetric space of $(\textrm{GL}(1,\R)^+)^m$, its derivative is $\sum_i \psi_i',$ where $\psi_i'=\phi_i'+\overline{\phi}_i'.$
    \item The symmetric space bundle $M\to C$ of metrics on this flat bundle is an affine bundle and the map $\tilde{f}$ determines a harmonic section, say $f':C\to M$. The apartment bundle $M_\phi \to S-B$ and its section $f$ are the objects one gets by taking the quotient of $(M,f')$ by the group $\Gamma.$
\end{itemize}

\end{subsection}

\begin{subsection}{Convergence under the $\R^+$ action}
We return to the setup from the very beginning of section 3.2, considering the $\R^+$-action on the moduli space of stable Higgs bundles on a closed Riemann surface $S$. Let $(E,\overline{\partial}_E,\phi)$ be a minimal slope-stable and generically semisimple Higgs bundle. Taking $R\mapsto (E,\overline{\partial}_E,R\phi),$ the non-abelian Hodge correspondence outputs a family of minimal harmonic bundles $R\mapsto (E,D_R,H_R)$ with holonomy representations $\rho_R.$ 

\begin{remark} 
The special toral bundle $F_{R\phi}^0$ of $(E, R\phi)$ is the same for every $R$, and hence the space of variations of $f_{R\phi}$ is well-defined independent of $R$.
The apartment bundle $M_{R\phi}$ and the harmonic section $f_{R\phi}$ are rescaled by $R$, and the stability form is rescaled by $R^2$. In particular the index of the stability form of $f_{R\phi}$ is independent of $R$.
\end{remark}

Recall that the statement of Theorem B is that the liminf as $R\to\infty$ of the indices of the minimal harmonic bundles $(E,D_R, H_R)$ is at least the index of the limiting object $(M_\phi, f)$. Let $Q_{H_R}$ be the stability form of $(E, D_R, H_R)$, and let $Q_f$ be the stability form of $(M_\phi, f)$. Recall that we can identify $\mathrm{Var}(H_R)$ with the space of smooth sections of $\mathrm{End}^0_{H_R}(E)$, whereas $\mathrm{Var}(f)$ is the space of smooth sections of $\R F_\phi^0$ with compact support on $S - B$. In order to relate the operators $Q_{H_R}$ to $Q_f$, we need to turn variations of $f$ into variations of $H_R$. Fortunately, there is a completely natural way to do this.  If $X$ is a variation of $f$, it is in particular a section of $\mathrm{End}^0(E)$, so we can simply project to $\mathrm{End}^0_{H_R}(E)$. Namely, let $X_R$ be the section of $\mathrm{End}_{H_R}^0(E)$ defined by
\begin{equation} \label{eqn: XR}
X_R = \frac{X + X^{*_R}}{2}
\end{equation}
where $*_R$ is an abbreviation of $*_{H_R}$. We will prove 

\begin{prop} 
\label{prop:limddota}
For any minimal slope-stable traceless Higgs bundle $(E,\overline{\partial}_E,\phi)$ with $\phi$ generically semisimple, and any $X \in \mathrm{Var}(f)$, we have
$$
\lim_{R\to\infty} Q_{H_R}(X_R) = Q_f(X).
$$
\end{prop}
Once we have proved this, the proof of Theorem B, at least in the case of $\mathrm{PGL}(n,\C)$, follows immediately by applying the following elementary lemma to $Q_R = Q_{H_R}$ and $Q = Q_f$.

\begin{lem}\label{indexlem} For $R \in \R^+$, let $V_R$ be a family of real vector spaces with quadratic forms $Q_R$. If $V$ is a real vector space with quadratic form $Q$, and for each $R$, there is a linear map $X \mapsto X_R$ from $V$ to $V_R$ such that $\lim_{R \to \infty} Q_R(X_R) = Q(X)$, then $\liminf_{R \to \infty} \mathrm{Ind}(Q_R) \geq \mathrm{Ind}(Q)$.
\end{lem}

\begin{proof} For any finite $k \leq \mathrm{Ind}(Q)$, let $W$ be a $k$-dimensional subspace of $V$ on which $Q$ is negative definite. For each $R$, $X \mapsto Q_R(X_R)$ is a quadratic form on the finite dimensional space $W$, so the pointwise convergence to $Q(X)$ implies locally uniform convergence on $W$. Hence, for large enough $R$, $Q_R(X_R)$ is also negative definite. For such $R$, the dimension of $\{X_R | X \in W\}$ must also be $k$, since otherwise we would have $Q_R(X_R) = 0$ for some vector $X$, but this is impossible as $Q_R(X_R)$ is negative definite. Hence, for such $R$, we have $\mathrm{Ind}(Q_R) \geq k$. Since $k \leq \mathrm{Ind}(Q)$ was arbitrary, $\mathrm{Ind}(Q_R) \geq \mathrm{Ind}(Q)$.
\end{proof}

\begin{remark}
    As we have indicated in Remark \ref{rem: index behaviour}, it is unclear to us in general how $Q_{H_R}$ and its index depend on $R$, apart from the limiting lower bound.
\end{remark}

Proving Proposition \ref{prop:limddota} requires an analysis of the behaviour of the minimal harmonic bundles $(E,D_R,H_R)$ as $R$ tends to infinity, and for this we turn to Mochizuki's work \cite{Mo}. In order to slightly simplify the statements of results below, as well as the proof of Proposition \ref{prop:limddota}, we fix a metric $\sigma_0$ on $S-B$ to define the norm of $1$-forms. 

\begin{thm}[Proposition 2.3 and 2.10 in \cite{Mo}]\label{thm: mochizuki}
Suppose that over a domain $U$ in $S$, the bundle $E$ splits as $E = \bigoplus_i E_i$ and $\phi$ acts on $E_i$ by multiplication by a $1$-form $\phi_i$. Let $\pi_i$ be the projection to $E_i$ determined by the splitting, and let $(\pi'_i)_R$ be the $H_R$-orthogonal projection to $E_i$.
If $U'$ is any relatively compact subdomain of $U$, then there are constants $C_1$ and $\epsilon$ depending on $U$, $U'$, the rank of $E$, and the $1$-forms $\phi_i$ such that
\begin{equation}\label{moch1}
|\pi_i - (\pi'_i)_R|_{H_R} < C_1 e^{-\epsilon R} \end{equation} and 
\begin{equation}\label{moch2}
    |\partial_{H_R} \pi_i|_{H_R, \sigma_0} < C_1e^{-\epsilon R}
\end{equation}
on $U'.$ 
\end{thm}

Mochizuki describes the constant $\epsilon R$ in terms of the gaps between the eigenvalues of $R\phi_i$. All we need is the linear dependence on $R$. By Remark \ref{rmk: polystable}, if $E$ is polystable, the constants $C_1$ and $\epsilon$ do not depend on the choice of harmonic metric $H_R$.

The second point is stated in Mochizuki with the additional assumption that the $E_i$ are one-dimensional. But the only place he uses it is in the short proof of his Theorem 2.9, and the only property he uses is that $\phi$ commutes with endomorphisms of $E_i$. This is true in our case because we specified that $\phi$ acts on $E_i$ by a multiple of the identity. This slight generalization will be useful for applying Mochizuki's results to groups other than $\textrm{PGL}(n,\mathbb{C})$ in section \ref{GHigg}.

We recall that if $Z$ is a section of the toral bundle $F^0_\phi$, then $Z^\dagger$ and $dZ = \partial Z + \dbar Z$ are the conjugate and derivative respectively of $Z$ with respect to the real structure and flat connection on $F^0_\phi$.  

In the following corollary, so as to reduce the number of subscripts, we abbreviate $H_R$ by $R$ in the adjoint, the norm, and the connection.

\begin{cor} \label{cor: estimates} Let $(E, \dbar_E, \phi)$ be a generically semisimple stable traceless Higgs bundle with critical set $B$ and toral bundle $F_\phi^0 \subset \mathrm{End}^0(E)|_{S-B}$. Let $Z$ be a smooth section of $F_\phi^0$ on $S-B$ and let $U$ be a subdomain of $S-B$ with compact closure. Then there are constants $C_2, \epsilon > 0$ depending on $(E, \dbar_E, \phi), U, |Z|$, and $|\dbar Z|_{\sigma_0}$ such that
\begin{enumerate}[label=(\roman*)]
    \item $|Z^{*_R} - Z^\dagger|_R$
    \item $|\partial_R Z^{*_R} - \partial Z^\dagger|_{R,\sigma_0}$
    \item $|\partial_R Z - \partial Z|_{R,\sigma_0}$, and
    \item $||Z|_R - |Z||$
\end{enumerate}
are all bounded by $C_2e^{-\epsilon R}$ on $U$.
\end{cor}

\begin{proof}
First suppose that $(E,\overline{\partial}_E, \phi)$ satisfies the splitting condition of Theorem \ref{thm: mochizuki} on a neighborhood of $U$, and write $Z = \sum_i Z_i \pi_i$.

\begin{enumerate}[label=(\roman*)]
    \item Let $(\rho_i)_R = \pi_i - (\pi_i')_R$. Since $(\pi_i')_R$ is self-adjoint with respect to $H_R$, we have $(\rho_i)_R^{*_R} = \pi_i^{*_R} - (\pi_i')_R$ and so
    \[
            |\pi_i - \pi_i^{*_R}|_R = |(\rho_i)_R - (\rho_i)_R^{*_R}|_R \leq 2 |(\rho_i)_R|_R \leq 2C_1 e^{-\epsilon R}
    \]
    by equation (\ref{moch1}) from Theorem \ref{thm: mochizuki}.
    So for some constant $C_2$,
    \[
    |Z^{*_R} - Z^\dagger|_R \leq \sum_i |Z_i||\pi_i^{*_R} - \pi_i|_R \leq C_2 e^{-\epsilon R}.
    \]
    \item Writing $Z=\sum_i Z_i\pi_i$ and using $\partial_R \pi_i^{*_R}=0,$ we see
    \[
    |\partial_R Z^{*_R}-\partial Z^\dagger|_{R,\sigma_0}\leq\sum_i |\partial \overline{Z}_i|_{\sigma_0}|\pi_i^{*_R}-\pi_i|_R\leq C_2e^{-\epsilon R}.
    \]
    \item Writing $Z = \sum_i Z_i \pi_i$ yet again,
    \[
    |\partial_R Z - \partial Z|_{R,\sigma_0} \leq \sum_i |Z_i| |\partial_R \pi_i|_{R,\sigma_0} \leq C_2 e^{-\epsilon R}
    \]
    where the last inequality is equation (\ref{moch2}) from Theorem \ref{thm: mochizuki}.
    
    \item We have $|Z|_R^2 - |Z|^2 = \sum_{i,j} Z_i\overline{Z_j}(\tr(\pi_i\pi_j^{*_R}) - \tr(\pi_i\pi_j))$, and
    \[
    |(\tr(\pi_i\pi_j^{*_R}) - \tr(\pi_i\pi_j))|_R \leq |\pi_i|_R|\pi_j^{*_R}- \pi_j|_R.
    \]
    
    Recall $r_i = \mathrm{rank}(E_i).$ Then expanding the inner product, we find
    \[
    |\pi_i - \pi_i^{*_R}|^2_R = 2 |\pi_i|^2_R - 2r_i,
    \]
    from which we conclude that $|\pi_i|_R$ is uniformly bounded in $R$. Thus $$|\pi_i|_R|\pi_j^{*_R} - \pi_j|_R \leq C_2e^{-\epsilon R}$$
    for some value of $C_2$.
\end{enumerate}

In general, we can cover $U$ by finitely many sets $U'_\alpha$ such that the conditions of Theorem \ref{thm: mochizuki} are satisfied on a neighborhood of $U'_\alpha$, and take $\epsilon$ (resp. $C_2$) to be the minimum (resp. maximum) of the value for each $U'_\alpha$. 
\end{proof}
We may now prove Proposition \ref{prop:limddota}.

\begin{proof}[Proof of Proposition \ref{prop:limddota}]

From Proposition \ref{stabformharm},
\[
Q_{H_R}(X_R) = 2\int_S |\partial_{R} X_R|_{R}^2 - \frac{2|\langle\partial_{R} X_R,\phi\rangle|^2}{|\phi|_{R}^2} + R^2|[X_R,\phi]|_{R}^2.
\]
On the other hand, from \eqref{eqn:ddota in F} we have
\[
Q_f(X) = 2\int_S |\partial X|^2 - \frac{2 |\langle \partial X, \phi\rangle|^2}{|\phi|^2}.
\]
We prove uniform convergence of each of the first two terms, and uniform convergence of the third term to zero. Starting with the first,
\[
\big||\partial_R X_R|_R - |\partial X|_R\big|\leq
|\partial_R X_R - \partial X|_R \leq \bigg|\frac{\partial_R X}{2} - \frac{\partial X}{2}\bigg|_R + \bigg|\frac{\partial_R X^{*_R}}{2} - \frac{\partial X}{2}\bigg|_R,
\]
which is $O(e^{-\epsilon R})$ by (ii) and (iii) of Corollary \ref{cor: estimates}. Since
$$||\partial_R X_R|_R - |\partial X||\leq \big||\partial_R X_R|_R - |\partial X|_R\big|+ \big||\partial X|_R - |\partial X|\big|,$$
combining with (iv) yields that $||\partial_R X_R|_R - |\partial X||$ is also $O(e^{-\epsilon R}).$

For the second term we have exponential convergence of the denominators by (iv) and for the numerators,
\[
\big||\langle\partial_R X_R, \phi\rangle| - |\langle \partial X, \phi\rangle|\big| \leq |\langle\partial_R X_R - \partial X,\phi\rangle| \leq |\partial_R X_R - \partial X|_R|\phi|_R.
\]
The term $|\phi|_R$ is uniformly bounded by (iv), and $|\partial_R X_R - \partial X|_R$ is exponentially decreasing by the previous equation.

Finally, since every section of $F_\phi^0$ commutes with $\phi$, we have $[X, \phi]=0$ and so
\[
|[X_R, \phi]|_R = |[\frac{X^{*_R} - X}{2}, \phi]|_R \leq |X^{*_R} - X|_R|\phi|_R.
\]
For the last inequality, we bounded the commutator by twice the norm of the product, and then used that Frobenius norms are sub-multiplicative.
Since $|\phi|_R$ is uniformly bounded by (iv) and $|X^{*_R} - X|_R$ is exponentially small by (i), this term goes to zero (despite the factor of $R^2$, which is dwarfed by the exponential decay).

\end{proof}

\begin{proof}[Proof of Theorem B for $\mathrm{PGL}(n,\C)$] Apply Lemma \ref{indexlem} to $V_R=\textrm{Var}(H_R),$ $V=\textrm{Var}(f)$, $Q_R=Q_{H_R}$, and $Q=Q_f$. The assumptions of Lemma \ref{indexlem} hold by Proposition \ref{prop:limddota}. 

\end{proof}

\end{subsection}

\begin{subsection}{Convergence for $G$-Higgs bundles}\label{GHigg}

This subsection is not necessary for the proof of Theorem A for $\mathrm{PSL}(n,\R)$.

Fix an admissible triple $(G,K,\nu)$. Let $(P, A^{0,1}, \phi)$ be a stable $G$-Higgs bundle that is minimal with respect to $\nu$. Suppose that $\phi$ is generically semisimple.

For each $R>0$, let $(M_R,h_R)$ be the associated minimal $G$-harmonic bundle. The left hand side of Theorem B is about the index of $h_R$ with respect to the space of smooth sections of $h_R^*T^{\textrm{vert}}M_R.$ Recall that this is identified with $P^c \times_K \p$.  In order to apply the results of the previous section, we use Proposition \ref{essfaithex} to fix an essentially faithful representation $\sigma = \prod_i \sigma_i: G \to \prod_i \mathrm{SL}_{n_i}(\C)$ of $G$ and constants $a_i$ such that the induced map of symmetric spaces is a totally geodesic isometry. By Proposition \ref{prop: tot geodesic index}, since the product of symmetric spaces $\prod_i\Met(\C^{n_i})$ is nonpositively curved, we should be able to express the index of $h_R$ using this isometry. In terms of harmonic bundles, this expectation is realized as follows.

For each $i$, let $(E_i, \dbar_{E_i}, \phi_i)$ be the Higgs bundle associated to $(P,A^{0,1},\phi)$ by $\sigma_i$. It is possible that $E_i$ is not stable, but only polystable; if so, fix once and for all a relative scaling of its stable factors as in Remark \ref{rmk: polystable}, determines for each $R \in \R^+$ a harmonic bundle $(E_i, D^R_i, H^R_i)$. The bundle $\Met(E_i)$ inherits the metric $a_i \tr$, and for each $R$ the induced map of flat Riemannian bundles from $M_R$ to $\prod_i \Met(E_i)$ is a fiberwise totally geodesic isometry, sending $h_R$ to $\prod_i [H^i_R]$. In particular the section $[H_R] := \prod_i [H^i_R]$ is minimal, and its index is the same as the index of $h_R$.

Let $X_R$ be a variation of $[H_R]$, which we can write as $X_R = \sum_i X^R_i$, where each $X^R_i$ is a section of $\mathrm{End}^0_{H^i_R}(E_i)$. The following two propositions are easy generalizations of Propositions \ref{stabformharm} and \ref{prop:limddota} respectively.

\begin{prop} The stability form of $H_R$ applied to $X_R$ is
$$Q_{H_R}(X_R) = \int_S \Big (\sum_i a_i |\partial_{H_R^i} X^R_i|_{H_R^i}^2\Big) - \frac{2\Big | \sum_i a_i\langle X^R_i,\phi_i\rangle\Big |^2}{\sum_i a_i |\phi_i|_{H_R^i}^2} + R^2\Big (\sum_i a_i|[X^R_i,\phi_i]|_{H_R^i}^2\Big).$$
\end{prop}

\begin{proof}
This is a consequence of the formula (\ref{stabform}). The curvature splits over the direct sum, so we can apply the same manipulations from Proposition \ref{stabformharm}
to get the last term.
\end{proof}
Now let $B^\sigma$ be the union of the critical sets of each $\phi_i$, and define $F_\phi^\sigma$ to be the fiber product of the toral bundles $F^0_{\phi_i}$ over $S-B^\sigma$ for each $\phi_i$. Each $F^0_{\phi_i}$ inherits the rescaled complex metric $a_i \tr$ from $\mathrm{End}^0(E_i)$, and together these give a parallel metric on $F_\phi^\sigma$. Let $M_\phi^\sigma$ be the fiber product of the apartment bundles $M_{\phi_i}$ for each $\phi_i$, and $f^\sigma$ the product of the canonical sections $f_i$. Then the minimality of $\phi$ with respect to $\nu$ implies that $f^\sigma$ is minimal with respect to the metric on $M_\phi^\sigma$. Let $Q_{f^\sigma}$ be its stability form.

A variation $X$ of $f^\sigma$ is a compactly supported section of $F^\sigma_\phi$. Given $X$, for each $R$, let $X^i_R$ be the $H^i_R$-self-adjoint part of $X^i$ as in \eqref{eqn: XR}, and set $X_R = \sum_i X^i_R$. Then $X_R$ is a variation of $[H_R]$, and we have

\begin{prop} \label{prop: G lim ddota} For any section $X \in \mathrm{Var}(f^\sigma)$,
$$
\lim_{R\to\infty} Q_{H_R}(X_R) = Q_{f^\sigma}(X).
$$
\end{prop}

\begin{proof} Exactly parallel to formula \eqref{eqn:ddota in F}, the stability form $Q_{f^\sigma}$ is given by
$$Q_{f^\sigma}(X)= \int_S \Big (\sum_i a_i |\partial X_i|^2\Big) + \frac{\Big | \sum_i a_i\langle X_i,\phi_i\rangle\Big |^2}{\sum_i a_i |\phi_i|^2}.$$
In the proof of Proposition \ref{prop:limddota}, we showed that each term of $Q_{H_R}(X_R)$ converges to the corresponding term of $Q_f(X)$, and even showed the separate convergence of the numerator and denominator of the second term. Hence, the same argument gives the convergence of $Q_{H_R}(X_R)$ to $Q_{f^\sigma}(X)$. 
\end{proof}

Proposition \ref{prop: G lim ddota} together with the index Lemma \ref{indexlem} therefore implies that 
$$
\liminf_{R \to \infty} \mathrm{Ind}(H_R) \geq \mathrm{Ind}(f^\sigma).
$$
Since $\mathrm{Ind}(h_R) = \mathrm{Ind}(H_R)$, we have almost proved Theorem B. In fact, we have already achieved the main point, which is to give a useful lower bound for the index of $h_R$. To complete the proof, it remains only to define the $G$-apartment bundle $M^G_\phi$ and its minimal section $f$, and show that $f$ has the same index as $f^\sigma$. 

\begin{subsubsection}{$G$-apartment bundles} \label{sec: G apartment}
Our definition of $M^G_\phi$ is a natural generalization of the apartment bundle for a classical Higgs bundle, corresponding to the group $\mathrm{PGL}(n, \C)$. We first define the critical set $B$ and the $G$-toral bundle $F^G_\phi$.  

At each point we have a Jordan decomposition of $\phi=fdz$ into its commuting semisimple and nilpotent parts as $\phi=\phi_s+\phi_n,$ where $\phi_s=f_sdz$ and $\phi=f_ndz$. Recall that we are working with a generically semisimple $G$-Higgs bundle $(P,A^{0,1},\phi),$ and hence $\phi_n$ is zero on the complement of a discrete set. 
\begin{defn}
The critical set $B\subset S$ of $(P,A^{0,1},\phi)$ is the subset of points $p\in S$ on which the rank of the center of the centralizer of $\phi_s$ in $\textrm{ad}\g^{\C}$ intersected with $\textrm{ad}\p^{\C}$ is strictly less than its maximum value on $S$. 
\end{defn}
Since $\phi$ is holomorphic, $B$ is discrete. Over any contractible open subset $U\subset S$ on which $P$ has been trivialized as $U \times K^\C$, we view the Higgs field $\phi$ as a $\p^\C$-valued $1$-form. We show $\phi=\phi_s$ on $S-B.$ Toward this, we recall a known stratification of $\p^{\C}$. 

For more details on the discussion below, see \cite[sections 3.1 and 3.2]{GL}. Fix a maximal toral subalgebra $\aaa$ of $\p$, with roots $\Delta$ and Weyl group $W$. Let $\mathcal{P}$ be the poset of subsets $A\subset \Delta$ that are closed under $\mathbb{Q}$-linear combinations, ordered by inclusion. For each $A\in \mathcal{P}$ set
 $$V_A=\{X\in \aaa^{\C}: \alpha(X)=0 \textrm{ for all } \alpha\in A\},$$ and let $V_A'\subset V_A$ be the open subset that is not contained in any smaller subspace of the form $V_B$, $B\in\mathcal{P}$. Since $W$ permutes the roots, it acts on $\mathcal{P}$, and it also acts on $\{V_A:A\in\mathcal{P}\}$ by $w\cdot V_A = V_{w\cdot A},$ the latter action satisfying $(w\cdot V_A)'=w\cdot V_A'.$ Thus, there is a $W$-invariant stratification of $\aaa^\C$ indexed by the quotient $\mathcal{P}/W$ whose strata are the $W$-orbits of $V_A'$ for $A \in \mathcal{P}$. Since it is $W$-invariant, in conjunction with the Chevalley map $\p^\C \to \aaa^\C/W$ it determines a $K^\C$-invariant stratification of $\p^\C$. 
 
 Let $S_{[A]}$ be the stratum in $\p^\C$ corresponding to the equivalence class $[A]$ in $\mathcal{P}/W$ determined by $A\in\mathcal{P}$. If $G = \mathrm{PGL}(n,\C)$, then $\mathcal{P}/W$ is in bijection with partitions of $n$, and $S_{[A]}$ is the set of matrices whose generalized eigenspaces have dimensions given by the partition. The $K^{\C}$-invariant stratification of $\p^{\C}$ studied in \cite[section 3.2]{GL} is a refinement of ours that takes into account nilpotent parts as well.  
\begin{lem}
    If $X \in V_A'$, then the center of the centralizer of $X$ is $V_A$.
\end{lem}
\begin{proof}
Centers of centralizers can be understood using the stratification. Let $\g^{\C}=\aaa^{\C}\oplus_{\alpha\in \Delta} \g_{\alpha}^{\C}$ be the weight decomposition according to the adjoint action of $\aaa$, specifically 
 \begin{equation}\label{eq: weightspace}
     \g_{\alpha}^{\C} = \{H_\alpha \in \g_{\alpha}^{\C}: [X,H_\alpha]=\alpha(X) \textrm{ for all }X\in \aaa\}.
 \end{equation} 
For $A\in\mathcal{P}$ let $$\mathfrak{p}_A^{\C}=\bigoplus_{\alpha \in A}(\g_{\alpha}^{\C}\oplus \g_{-\alpha}^{\C})\cap \p^{\C}.$$ According to \cite[Lemma 3.2.1]{GL}, if $X\in V_A'$, then the centralizer of $X$ in $\p^{\C}$ is $\p^X =\aaa^{\C}\oplus \p_A^{\C}$. Using the definition (\ref{eq: weightspace}), it is immediate that the center of the centralizer of $X$ is $V_A$.
\end{proof}

\begin{lem}\label{lem: phi = phis}
    The generically semisimple Higgs field $\phi$ is semisimple on all of $S-B$.
\end{lem}

 \begin{proof}
Since the stratification above is $K^{\C}$-invariant, it induces via local trivializations a stratification of $\textrm{ad}\p^{\C}$. Explicitly, choosing a local trivialization $U\times \p^{\C}$, the stratum corresponding to $[A]$ intersected with $U\times \p^{\C}$ is $U\times S_{[A]}$. By holomorphicity of $f$, the set of $[A]_0 \in \mathcal{P}/W$ such that $f_s \in \overline{S_{[A]_0}}$ everywhere has a unique maximal element $[A]$. Let $S^s_{[A]}$ be the locus of semisimple elements of $S_{[A]}$. It follows from \cite[section 3.2]{GL} that the closure of $S_{[A]}^s$ is contained in $S_{[A]}^s \cup \bigcup_{A_0 < A} S_{[A_0]}$. Therefore, the locus of non-semisimple elements of $S_{[A]}$ is open in $\overline{S_{[A]}}$. Since $f$ is generically semisimple, it cannot meet this open set. Hence $f$ is always valued in $S_{[A]}^s \cup \bigcup_{[A]_0 < [A]} S_{[A]_0}$.

Since the stratification depends only on the image under the Chevalley map, $f_s \in S^s_{[A]_0}$ if and only if $f \in S_{[A]_0}$. By the characterization of the center of the centralizer given above, its dimension jumps whenever $f \in \bigcup_{[A]_0 < [A]} S_{[A]_0}$. Hence $S-B$ is exactly the set on which $f \in S^s_{[A]}$, and in particular $f$ is semisimple on all of $S-B$.
 \end{proof}

\begin{defn}
    The $G$-toral bundle $F_\phi^G$ is the vector bundle over $S-B$ whose fiber over a point is the intersection of $\mathrm{ad}\p^\C$ with the center of the centralizer of $\phi_s$ in $\mathrm{ad}\g^\C$. 
\end{defn}
To see that $F_\phi^G$ defines a holomorphic vector bundle, we can use the stratification. Let $U\subset S$ be a contractible open subset on which $P$ is trivialized to $U \times K^\C$, let $z\in U - B$, and let $S_A$ be the stratum determined by $\phi$. That means that there is an element $k \in K^\C$ conjugating $\phi$ at $z$ into $V_A'$. Since $V_A'$ is open in its stratum of $\aaa^\C$, we can extend $k$ to a gauge transformation, still called $k$, in a neighborhood $U_z$ of $z$ conjugating $\phi$ into $V_A'$. Thus, we can trivialize $F_\phi^G|_{U_z}$ as $U_z\times V_A$. By Proposition 2.20.18 in \cite{Eb}, which we already used in the proof of Lemma \ref{lem: same action}, the gauge transformation is independent of the choice of $k$. It follows that it is also independent of the original trivialization of $P$, since any two trivializations also differ by a gauge transformation. Since the trivialization to $U_z\times V_A$ is independent of choices, doing the same over every point of $S-B$ gives $F_\phi^G$ the structure of a holomorphic vector bundle.

The local trivializations to $V_A$ give the $G$-toral bundle $F_\phi^G$ a canonical flat connection and real structure. In any local identification of $F^G_\phi$ with $V_A\subset \aaa^{\C}$, the roots of $\aaa^\C$ are real and parallel. As $F_\phi^G$ is a subbundle of $\mathrm{ad}\p^\C$, it inherits the $K^{\C}$-invariant metric $\nu$, which is parallel with respect to the flat connection. 
By Lemma \ref{lem: phi = phis}, as in section 3.2, $\phi$ is a holomorphic $F^G_\phi$-valued $1$-form, so by integrating its real part, we obtain an affine flat Riemannian bundle $M^G_\phi$ with section $f$. This is the $G$-apartment bundle. The proof of Theorem B is completed by

\begin{prop} $\mathrm{Ind}(f) = \mathrm{Ind}(f^\sigma).$
\end{prop}
Recall that $\sigma=\prod_i \sigma_i : G\to \prod_i \textrm{SL}_{n_i}(\C)$ is an essentially faithful representation, giving rise to the product of toral bundles $F_\phi^\sigma = \prod_i F_{\phi_i}^0$ with product of apartment bundles $M_\phi=\prod_i M_{\phi_i}$ and minimal section $f^\sigma$.
\begin{proof}
The representation $\sigma$ defines an embedding $\textrm{ad}\sigma:\mathrm{ad}\p^\C\to\prod_i \mathrm{End}^0(E_i),$ which is compatible with the complex linear metrics on each.
Recall that $B^\sigma$ is the critical set of the Higgs field $\textrm{ad}\sigma(\phi)$. We will show that $B\subset B^\sigma,$ and over $S - B^\sigma$, there is a flat Riemannian affine bundle $M_\phi'$ with canonical harmonic section $f'$ and a diagram of fiberwise totally geodesic linear isometric embeddings
\[
M^G_\phi \leftarrow M_\phi' \rightarrow M^\sigma_\phi
\]
sending the canonical sections to one another. Propositions \ref{prop: log cutoff} and \ref{prop: tot geodesic index} will then show that the index of $f$ is equal to that of $f'$ and that the index of $f^\sigma$ is equal to that of $f',$ from which the result will follow.

Define $F_\phi'$ over $S-B^\sigma$ to be $\textrm{ad}\sigma^{-1}(F_\phi^\sigma)$. It is tautological that $F_\phi'\subset F_\phi^G.$ We will first show that $F_\phi'$ is a flat real subbundle of $F_\phi^G$, and therefore also inherits the metric $\nu$. Second, we will show that $\textrm{ad}\sigma|_{F_\phi'}$ preserves these structures.

Fix the maximal toral subalgebra $\aaa$ of $\p$ as above. Say that an element of $\aaa^\vee$ is a root of $\sigma$ if it is the difference of two weights of the same factor $\sigma_i$. The roots of $\sigma$ are always invariant by the Weyl group $W$. Futhermore they contain the roots of $\g$, since if $e_\alpha$ is a root vector for the root $\alpha$, and $v_\lambda$ is a vector in $\C^{n_i}$ of weight $\lambda$ such that $e_\alpha \cdot v_\lambda \neq 0$ then $e_\alpha \cdot v_\lambda$ is a vector of weight $\lambda + \alpha$. Hence, the roots of $\sigma$ define a $W$-invariant refinement of the previous stratification of $\aaa^\C$, and in turn a refinement of the $K^{\C}$-invariant stratification of $\p^\C$. It is not hard to see from the definition of toral bundles that $F'_\phi$ is built from this stratification in exactly the same way that $F^G_\phi$ is built from the original stratification; namely, $B^\sigma$ is the finite set at which $\phi_s$ is in a smaller stratum, and with respect to a local trivialization near $z \in S - B^\sigma$ such that $\phi_s$ is contained in $\aaa^\C$, the fiber of $F'_\phi$ at $z$ is the smallest subspace of $\aaa^\C$ in the $W$-invariant stratification containing $\phi_s(z)$. Since the stratification is a refinement of the original, $B \subset B^\sigma$. Since $F_\phi'$ is cut out by roots of $\sigma,$ which are locally constant and real, $F_\phi'$ is indeed a flat real subbundle.

Since $\sigma$ is an isometry by construction, $\textrm{ad}\sigma:F_\phi'\to F_\phi^\sigma$ preserves the metrics. Up to conjugation, we can assume the map $\sigma:\g\to \prod_i \mathfrak{sl}(n_i,\C)$ sends $\aaa$ to real diagonal matrices. In local trivializations, this is a model for $\textrm{ad}\sigma:F_\phi'\to F_\phi^\sigma$, and hence $\textrm{ad}\sigma: F_\phi'\to F_\phi^\sigma$ is real and flat, as we claimed we would show.

Totally analogous to the construction of the apartment and $G$-apartment bundles, the section $\textrm{ad}\sigma(\phi)$ of $\textrm{ad}\sigma(F_\phi')$ determines the flat Riemannian affine bundle $M_\phi'$ with canonical harmonic section $f'$. $M_\phi'$ is naturally a totally geodesic subbundle of $M_\phi^G,$ and identifies with a totally geodesic subbundle of $M_\phi^\sigma.$ Applying Propositions \ref{prop: log cutoff} and \ref{prop: tot geodesic index}, as discussed above, shows $\textrm{Ind}(f)=\textrm{Ind}(f^\sigma).$
\end{proof}

\end{subsubsection}

\end{subsection}
\end{section}

\begin{section}{Unstable minimal surfaces and the Labourie Conjecture}
After recalling the Hitchin section, we prove Theorem A for $G=\textrm{PSL}(n,\R)$, $n\geq 4$, and Theorem 1.1. We then prove Theorem A in general and deduce Corollary A. 
\begin{subsection}{The Hitchin section}\label{sec: hitchin section}

Let $S$ be a Riemann surface structure on the closed surface $\Sigma_g$ and let $(G,K)$ be an admissible pair of rank $l$. Fix homogeneous generators $p_1, \ldots, p_l$ of $\pK$, and let $m_1, \ldots, m_l$ be their degrees. In our terminology, the Hitchin map with respect to the basis $p_1, \ldots, p_l$ sends a $G$-Higgs bundle $(P, A^{0,1}, \phi)$ to
\[
(p_1(\phi), \ldots, p_l(\phi)) \in \bigoplus_{i = 1}^l H^0(S, \mathcal{K}^{m_i}).
\]
The space $\bigoplus_{i = 1}^l H^0(S, \mathcal{K}^{m_i})$ is called the Hitchin base.

If $G$ is the adjoint form of a connected split real simple group equipped with a principal embedding of $\textrm{PSL}(2,\R)$, Hitchin constructs for each point $(\alpha_1, \ldots, \alpha_l) \in \bigoplus_{i = 1}^l H^0(S, \mathcal{K}^{m_i})$ a $G$-Higgs bundle $s(\alpha_1, \ldots, \alpha_l)$. He proves

\begin{thm}[Theorem 7.5 in \cite{Hi}]\label{hit}
With respect to the natural complex structure on the moduli space $\mathcal{M}_G(S)$ of $G$-Higgs bundles over $S$, $s$ is a holomorphic section of the Hitchin map. Furthermore, $s(\alpha_1, \ldots \alpha_l)$ is always stable and $s$ is an isomorphism onto a connected component of $\mathcal{M}_G(S)$.
\end{thm}

By $G$-NAH II, the image of the Hitchin section $s$ corresponds to a connected component of the representation variety $\mathrm{Rep}(\Sigma_g,G)$. This component of the representation variety is called the Hitchin component, which we write as $\mathrm{Hit}(\Sigma_g,G)$. For $G=\textrm{PSL}(n,\R),$ this definition coincides with the one given in section \ref{sec: labconjecture}. Similarly, as we mentioned in section \ref{sec: labconjecture}, for general $G$ one can also realize Hitchin components using the principle embedding from $\textrm{SL}(2,\R)\to G$ (see \cite{Hi}).

If $G$ is a product of adjoint forms of split real simple groups, we can take the generators $p_i$ to be the union of a basis of generators for each factor. Hitchin's construction gives a $G$-Higgs bundle for each factor. We extend the Hitchin section to products by simply taking the product of $G_i$-Higgs bundles for each factor $G_i$. Theorem \ref{hit} remains true for such groups $G$.

\begin{remark}
A metric $\nu$ on $(G,K)$ determines a degree two invariant polynomial on $\p^\C$ by $X\mapsto \nu(X^2)$. Since it should be clear in context, we write $\nu$ for the polynomial. When $G$ is simple, every degree two invariant polynomial is a multiple of $\nu.$
\end{remark}

\begin{remark}\label{rem: immersion}
If $(P,A^{0,1},\phi)$ lies in the image of the Hitchin section, then $\phi$ never vanishes \cite[section 5]{Hi}. It follows that minimal harmonic maps associated to Hitchin representations are immersions.
\end{remark}
\begin{remark}
    Note that for a complex Lie group $G^{\C}$, the Hitchin map is defined using polynomials in $\mathcal{O}(\g^{\C})^{G^{\C}}$, the ring of $G^{\C}$-invariant polynomials on $\g^{\C}.$ When $G$ is a split real form of $G^{\C}$, the restriction map $\mathcal{O}(\g^{\C})^{G^{\C}}\to \pK$ is an isomorphism. Hitchin really defined the Hitchin section as a section of the moduli space for the complex group, and then proved that the image consists of Higgs bundles for the split real form.
\end{remark}

To make things more concrete, we write out Higgs bundles describing the Hitchin section for a choice of Hitchin map for $G=\textrm{PSL}(n,\R)$. Viewing $\textrm{PSL}(n,\R)$ as a subgroup of $\textrm{PGL}(n,\C)$, we'll actually write out ordinary Higgs bundles that come from $\textrm{PGL}(n,\C)$-Higgs bundles.

Recall from section \ref{2.4} that $\p^\C\subset \g^\C=\mathfrak{sl}(n,\C)$ is the subset of complex symmetric matrices. One choice of basis for $\pK$ is the elementary symmetric polynomials $e_2,\dots, e_n$. For $a=\textrm{diag}(a_1,\dots, a_n)$ in the subalgebra $\aaa^{\C}\subset \p^{\C}$ of diagonal matrices, the $e_i$ are defined
$$e_2(a) = \sum_{1\leq i<j\leq n}a_ia_j, \hspace{1mm} e_3(a) = \sum_{1\leq i<j<k\leq n} a_ia_ja_k, \hspace{1mm} \dots \hspace{1mm} e_n(a) = \prod_{i=1}^n a_i.$$ By the real Chevalley restriction theorem, this defines $e_i$ on all of $\pK.$ The Hitchin base is then $\oplus_{i=2}^n H^0(S,\mathcal{K}^i)$. Note that if $q$ is a polynomial in $\C$ with roots $\lambda_1,\dots, \lambda_n,$ then 
\begin{equation}\label{charpoly}
    q(z)= \sum_{i=0}^n (-1)^{n-i}e_{n-i}(\lambda)z^i,
\end{equation}
where $\lambda=\textrm{diag}(\lambda_1,\dots, \lambda_n)$ and $e_1(\lambda)=\sum_{i=1}^n \lambda_i$ (which vanishes on our Lie algebra).

We choose a different set of generators for our Hitchin map, which is more in line with Hitchin's original presentation in \cite{Hi}. Let $E$ be the holomorphic vector bundle $E=\oplus_{i=1}^n \mathcal{K}^{\frac{n+1-2i}{2}}$ with standard holomorphic structure $\overline{\partial}_E$, and for $\alpha=(\alpha_2,\dots, \alpha_n)\in \oplus_{i=2}^n H^0(S,\mathcal{K}^i)$, set 
$$\phi_\alpha = 
\begin{pmatrix}
    0 & \alpha_2 & \alpha_3 &\dots & \alpha_{n-1} & \alpha_n & \\
    r_1 & 0 & \alpha_2 & \dots & \alpha_{n-2} & \alpha_{n-1} \\
    0 & r_2 & 0 & \dots &  \alpha_{n-3} & \alpha_{n-2}\\
    \vdots & \vdots & \vdots & \ddots & \ddots &\vdots\\
     0 & 0 & 0 & \dots &  0 & \alpha_2\\
    0 & 0 & 0 & \dots & r_{n-1} & 0 
\end{pmatrix}\ , $$
where $r_i = \frac{i(n-i)}{2},$ $1\leq i\leq n-1.$ To interpret the matrix above as an endomorphism of $E$, note that the $(i,j)$ component is supposed to be a section of $\textrm{Hom}( \mathcal{K}^{\frac{n+1-2j}{2}}, \mathcal{K}^{\frac{n+1-2i}{2}})\otimes \mathcal{K}$, which we're identifying with $\mathcal{K}^{j-i+1}$. Hitchin's method in \cite[\S3 and \S5]{Hi} shows that every $(E,\overline{\partial}_E,\phi_q)$ is stable.

There exists a set of homogeneous polynomials $p_1,\dots, p_{n-1}\in \pK$ such that $p_i(\phi_q)=\alpha_{i+1}$, but it's cumbersome to write them down explicitly (see the remark on page 4 of \cite{Hit2016}). Of course, they can be expressed in terms of the elementary symmetric polynomials. These polynomials $p_1,\dots, p_{n-1}$ are the standard choice for defining a Hitchin map, and with this choice, the Hitchin section $s$ associates $\alpha=(\alpha_2,\dots, \alpha_n)$ to the class of $(E,\overline{\partial}_E,\phi_\alpha)$ in the moduli space of Higgs bundles.  Note that the bundle $E$ has degree $0$ and the Hermitian metric obtained through NAH II makes $(E,\overline{\partial}_E+\partial^H + \phi + \phi^{*_H})$ a flat bundle, and not just a projectively flat bundle. Accordingly, the holonomy lifts to $\textrm{GL}(n,\C).$ Hitchin proves in \cite{Hi} that the holonomy in fact lives in $\textrm{SL}(n,\R).$

Finally, with Theorem \ref{thm: PSL(4,R)} in mind, we write down the generating polynomials explicitly for $n=4$. In this case, $r_1=r_3=\frac{3}{2}$ and $r_2=2$. Given complex numbers $a_2,a_3,a_4\in \C$, we compute a characteristic polynomial 
$$\det\Big(zI - 
\begin{pmatrix}
0 & a_2 & a_3 & a_4 \\
\frac{3}{2} & 0 & a_2 & a_3 \\
0 & 2 & 0 & a_2 \\
0 &0 & \frac{3}{2} & 0
\end{pmatrix}\Big )= z^4 - 5a_2 z^2 - 6a_3 z-\frac{9}{2}a_4 +\frac{9}{2}a_2^2.$$
Using the formula (\ref{charpoly}) we see that we should take $p_1=-\frac{e_2}{5},$ $p_2=\frac{e_3}{6}$, and $p_3=-\frac{2}{9}(e_4 - \frac{e_2}{25}).$
\end{subsection}

\begin{subsection}{Unstable minimal surfaces for Hitchin representations}\label{sec: unstable minimal surfaces}
In the beginning of this subsection we'll recall a few notations and results from the introduction. 
Let $(G,K,\nu)$ be an admissible triple with $G$ the adjoint form of a split real semisimple group, and let $N = G/K$. Let $\mathbf{T}_g$ be the Teichm{\"u}ller space of a surface $\Sigma_g$ of genus $g$. Recall that if $M$ is a flat Riemannian $N$-bundle with irreducible holonomy $\rho: \pi_1(\Sigma_g) \to G$, $G$-NAH I produces a harmonic section $h$. Let $\mathbf{E}_\rho(S)$ be the total energy of $h$ with respect to $\nu$, which we view as a function on $\mathbf{T}_g$. The function $\mathbf{E}_\rho$ is smooth \cite{Sle}. If $S$ is a critical point in $\mathbf{T}_g$, then the section $h$ is minimal with respect to $\nu$, and its index as a minimal map is equal to the index of $\mathbf{E}_\rho$ at $S$ \cite[Theorem 3.4]{Ej}. In the case that $\rho$ is in the Hitchin component, the result below follows from work of Labourie \cite[Theorem 1.0.3]{L1} (properness of energy for well-displacing representations), Guichard-Wienhard \cite[Theorem 1.7]{GuW} (Anosov representations are well-displacing), and Guichard-Wienhard-Labourie \cite{GLW} (Hitchin representations for all $G$ are Anosov).

\begin{thm}If $\rho$ is Hitchin, then $\mathbf{E}_\rho$ is proper on $\mathbf{T}_g$.
\end{thm}

Consequently, $\mathbf{E}_\rho$ admits a global minimum, which must in particular be a stable critical point. Therefore, every Hitchin representation has at least one stable equivariant minimal surface. Hence, to prove that there exists a Hitchin representation with more than one equivariant minimal surface, it suffices to produce a Hitchin representation with an unstable minimal surface. This is what we do below. But first, we state a well-known fact that is the key to the construction. Note that a $\R^n$-valued first cohomology class of $\Sigma_g$ defines an action of $\pi_1(\Sigma_g)$ on $\R^n$ by translations.

\begin{prop}\label{exun} Let $\Sigma_g$ be a closed surface of genus $g$ with $g \geq 3$. For any $n \geq 3$, there is a cohomology class $\beta \in H^1(\Sigma_g, \R^n)$, a conformal structure $S$ on $\Sigma_g$, and an unstable minimal map $\tilde{f}$ from a covering space $\tilde{S}$ of $S$ to $\R^n$, equivariant by the action of the Deck group determined by $\beta$.
\end{prop}
\begin{proof} Here is a very direct proof. In the case $g = 3$ and $n=3$, we take $\tilde{S}$ to be the triply periodic Schwarz P-surface, $\tilde{f}$ the inclusion, $S$ the quotient by $\Z^3$, and $\beta$ the coresponding cohomology class. The Weierstrass data of $\tilde{f}$ can be described as follows. The Riemann surface $S$ is defined by the equation $$w^2=z^8+14z^4+1,$$ where $z$ is a degree $2$ function on $S$ and $w$ is a degree $8$ function. We consider $S$ as a branched double cover of the complex plane with the coordinate $z$. The Weierstrass data is $$\phi_1 = \frac{(1-z^2) dz}{w}, \hspace{1mm} \phi_2 = \frac{i(1+z^2)dz}{2}, \hspace{1mm} \phi_3 = \frac{2zdz}{w},$$ and $\tilde{f}$ is then obtained by lifting these $1$-forms to the universal cover and integrating their real parts from a fixed basepoint. This surface has equivariant index one \cite{Ross}; to show that the index is at least one, it suffices to consider a normal variation of constant length. If $g > 3$, we can take a branched covering of the (quotient of the) Schwarz P-surface, and if $n > 3$, we can linearly isometrically map $\R^3$ into $\R^n$.
\end{proof}

In fact, there are many such examples, and, at least for large genus, many of very large index. For a discussion in the case $n=3$, see \cite[section 5.3]{MSS}.

By the usual correspondence between equivariant maps and flat bundles, we can also interpret this as a flat Riemannian bundle $M_{\beta}$, whose fibers are $\R^n$ and whose transition functions are translations given by the cohomology class $\beta$, together with a minimal section $f$. The vertical tangent bundle of $M_{\beta}$ is just the trivial $\R^n$-bundle. The equivariant index of $\tilde{f}$ is the same as the index of $f$. This proposition indicates that we should be looking for instability in Higgs bundles for which the $G$-toral bundle $F^G_\phi$ is trivial.

We first prove Theorem A for the main case of interest, $G=\textrm{PSL}(n,\R)$, $n\geq 4$. The proof is simple and distills the main ideas. 
\begin{proof}[Proof of Theorem A for $\textrm{PSL}(n,\R)$]
As above, let $\aaa^\C\subset \p^\C\subset \g^\C$ be the subalgebra of traceless diagonal matrices, 
$$
\aaa^\C = \{\textrm{diag}(a_1,\dots, a_n): a_i\in \C, \sum_{i=1}^n a_i = 0\}.
$$
Let $(\tilde{S},\beta, \tilde{f})$ be an unstable minimal map to $\R^{n-1}$ as in Proposition \ref{exun}. There is a $\textrm{SO}(n-1,\R)$-torsor of isometries from $\R^{n-1}$ to the subalgebra of real points $\aaa$, and we have the freedom to choose one, $\iota$, such that no two components of $\iota(\partial\tilde{f})$ agree at every point of $\tilde{S}.$ Let $\textrm{diag}(\phi_1,\dots,\phi_n) = \iota(\partial \tilde{f}).$ This is the holomorphic derivative of the section $f$ of $M_{\beta}$ in a particular trivialization, so we denote it by $\partial f$. From the choice of hyperplane, $e_1(\partial f)=0$, and since $f$ is minimal and the mapping from $\R^{n-1}$ to $\aaa$ is an isometry, $e_2(\partial f)=0$. By our choice of isometry $\iota$, $\partial f$ is generically regular semisimple. 

Let $p_1,\dots, p_{n-1}$ be any choice of homogeneous generators for $\pK$, with associated Hitchin section $s$. For each $i=1,\dots, n-1$, set $\alpha_i=p_i(\partial f)$, defining a point $(\alpha_2,\dots, \alpha_n)$ in the Hitchin base, and let $(E,\overline{\partial}_E,\phi)$ be the Higgs bundle associated to $s((\alpha_2,\dots, \alpha_n))$ via the standard representation. Since $\alpha_2 = 0$ and $p_2$ is a linear combination of $e_2$ and $e_1^2$, $p_2(\phi)=0,$ so the Higgs bundle $(E,\overline{\partial}_E,\phi)$ is minimal. Since the $p_i$'s generate $\pK$, $e_i(\phi)=e_i(\partial f)$ for all $i=2,\dots, n$, and hence by (\ref{charpoly}), the characteristic polynomials of $\phi$ and $\partial f$ agree. Hence, $\phi$ is not just generically regular semisimple, but it is globally diagonalizable on the complement of its critical set $B$, with eigen-$1$-forms $\phi_1,\dots, \phi_n$.

Since $\phi$ is globally diagonalizable on $S-B$, the extended toral bundle $F_\phi$---spanned by globally defined projections to the eigenspaces of $\phi$---is trivial. The toral bundle $F_\phi^0$ is then trivial as well. Equivalently, any small cameral cover is a copy of $S$ itself. After trivializing the vertical tangent bundle of the apartment bundle $M_\phi$ to $\R^n$, the holomorphic derivative of the minimal section agrees with $\partial f$ up to an isometry. By Theorem B, for $R$ large enough, the index of the minimal map $h_R$ associated to $(E,\overline{\partial}_E,\phi)$ is bounded below by the index of the limiting object, which is equal to the equivariant index of $\tilde{f}$, hence positive. This completes the proof. 
\end{proof}
Theorem \ref{thm: PSL(4,R)} is obtained just by specializing the proof above to $\textrm{PSL}(4,\R)$ and choosing a basis of polynomials that defines our Hitchin map.
\begin{proof}[Proof of Theorem \ref{thm: PSL(4,R)}]
Let $(\tilde{S},\beta,\tilde{f})$ be the data of an unstable minimal map to $\R^3$ as in Proposition \ref{exun}, such as that of the Schwarz P-surface. If $\tilde{f}=(\tilde{f}_1,\tilde{f}_2,\tilde{f}_3),$ the holomorphic derivative of $\tilde{f}_i$ descends to an abelian differential $\beta_i$ on $S,$ and the $\beta_i$'s satisfy $\beta_1^2+\beta_2^2+\beta_3^2=0.$ Choosing a linear isometry from $\R^3$ to the subalgebra $\aaa^{\C}$ from the proof above turns $(\beta_1,\beta_2,\beta_3)$ into a traceless matrix of abelian differentials $\textrm{diag}(\phi_1,\phi_2,\phi_3,\phi_4)$. We choose our isomorphism so that no two $\phi_i$'s are identically equal; this ensures that our matrix is generically regular semisimple. We define our Hitchin map and corresponding section using the polynomials from section \ref{sec: hitchin section}, $p_1=-\frac{e_2}{5},$ $p_2=\frac{e_3}{6}$, and $p_3=-\frac{2}{9}(e_4 - \frac{e_2}{25}).$ According to the proof of Theorem A for $\textrm{PSL}(n,\R),$ for $R>0$ sufficiently large, applying the non-abelian Hodge correspondence to any Higgs bundle in the Hitchin fiber over $$(0,R^3\alpha_3,R^4\alpha_4)=(0,\frac{1}{6}R^3(\phi_1\phi_2\phi_3+\phi_1\phi_2\phi_4 + \phi_1\phi_3\phi_4+\phi_2\phi_3\phi_4), -\frac{2}{9}R^4\phi_1\phi_2\phi_3\phi_4)$$ returns a Hitchin representation $\rho_R:\pi_1(S)\to \textrm{PSL}(4,\R)$ together with an unstable $\rho_R$-equivariant minimal map. Explicitly, consider the Higgs bundle in the Hitchin section with underlying holomorphic vector bundle $E=\mathcal{K}^{\frac{3}{2}}\oplus \mathcal{K}^{\frac{1}{2}}\oplus  \mathcal{K}^{-\frac{1}{2}}\oplus  \mathcal{K}^{-\frac{3}{2}}$, with standard holomorphic structure $\overline{\partial}_E$, and Higgs field $$\phi=
\begin{pmatrix}
    0 & 0 & \frac{1}{6} R^3\alpha_3 & -\frac{2}{9} R^4\alpha_4 \\
    \frac{3}{2} & 0 & 0 & \frac{1}{6} R^3\alpha_3 \\
 0 & 2 & 0 & 0 \\
    0 & 0 & \frac{3}{2} & 0 
\end{pmatrix}.$$ 
Then Labourie's conjecture fails for the representation to the $\textrm{PSL}(4,\R)$ obtained by applying NAH II to $(E,\overline{\partial}_E,\phi)$.
\end{proof}

The proof of the full Theorem A follows the $\textrm{PSL}(n,\R)$ case in an abstract setting. 

\begin{proof}[Proof of Theorem A]
Let $G$ be the adjoint form of a split real group of real rank $l \geq 3$, with a maximal compact $K$ and decomposition $\g = \kk \oplus \p$. Let $\aaa$ be a maximal toral subalgebra of $\p$ and let $\nu$ be an invariant metric on $G/K$. We set $(\tilde{S}, \beta, \tilde{f})$ to be an unstable minimal map to $\R^l$ from Proposition \ref{exun}. Choose an isometry between $\R^l$ and $\aaa$ (with the metric $\nu$) such that no root of $\aaa$ vanishes on the derivative $\partial \tilde{f}$ at every point of $\tilde{S}$. That is, $\partial \tilde{f}$ is generically in general position in the sense of section \ref{subsec: invariant polynomials} (since $\g$ is split, this is equivalent to being generically regular semisimple).

Let $M_{\beta}$ be the flat Riemannian bundle with translational holonomy given by $\beta$, with section $f$ descended from $\tilde{f}$. Using the chosen isometry, we can identify the vertical tangent bundle of $M_{\beta}$ with the trivial $\aaa$ bundle, so that $\partial f$ is a $(1,0)$-form valued in $\aaa^\C$. Since $f$ is minimal, $\nu((\partial f)^2) = 0$.

Let $(p_1, \ldots, p_n)$ be a homogeneous generating set for the invariant polynomials on $\p^\C$, and for each $i= 1, \ldots, n$, set $\alpha_i = p_i(\partial f)$. Then $(\alpha_1,\ldots, \alpha_n)$ is a point in the Hitchin base. Let $(P, A^{0,1}, \phi) = s((\alpha_1, \ldots, \alpha_n))$, where $s$ is the Hitchin section. Since $s$ is a section of the Hitchin map, every $K^\C$-invariant polynomial on $\p^\C$ takes the same values on $\phi$ as on $\partial f$. In particular, since $\nu$ is an invariant polynomial, the $G$-Higgs bundle $(P, A^{0,1}, \phi)$ is minimal with respect to $\nu$. Moreover, since $\partial f$ is in general position by the choice of isometry, and $s$ is a section of the Hitchin map, Lemma \ref{lem: reg semisimp} shows that $\phi$ is conjugate to $\partial f$ in any local trivialization of $P$, and hence that $\phi$ is also generically in general position (and generically regular semisimple). Let $B \subset S$ be the critical set of $\phi$, which is to say the set at which it is not in general position. Let $F^G_\phi$ be the toral bundle of $\phi$ over $S-B$ and $M^G_\phi$ be the $G$-apartment bundle of $\phi$ with minimal section $f^G$.

From Theorem $B$, for $R$ sufficiently large, the minimal harmonic map associated to $(P, A^{0,1}, R\phi)$ has index bounded below by the index of the minimal section $f^G$. Since the index of $f$ is positive, we are done if we can produce an isomorphism of flat Riemannian bundles between $M_{\beta}$ and $M_\phi^G$ that intertwines $f$ and $f^G$. It is even enough to produce an isomorphism between $\aaa^\C$ and $F^G_\phi$ that intertwines $\partial f$ and $\phi$, since the associated apartment bundles are then constructed by integration in the same way.

Fix a local trivialization of $P$ on $U \subset S-B$. By Lemma \ref{lem: same action}, in this trivialization every gauge transformation conjugating $\phi$ to $\partial f$ defines the same isomorphism from $F^G_\phi$ to $\aaa^\C$. Since the isomorphism is unique, these local isomorphisms patch together to give a global flat real isometric isomorphism between the bundle $F^G_\phi$ and the trivial bundle $\aaa^\C$ identifying $\phi$ with $\partial f$. This completes the proof of Theorem A.

\end{proof}
\begin{remark}\label{rem: genus 2}
    We restrict to $g\geq 3$ because every genus $2$ equivariant minimal surface in a Euclidean space is contained in a flat $2$-plane and hence stable (see \cite[section 5.3]{MSS}). To disprove the Labourie Conjecture for $g=2$ using Theorem B, one has to find a point in the Hitchin base with small cameral of genus larger than $2$ and equivariantly unstable minimal map in the sense of section \ref{sec: 3.2}.
\end{remark}

\end{subsection}

\begin{subsection}{Non-uniqueness of area minimizers}\label{sec: area minimizers}

To conclude the paper, we prove Corollary A, which states that area minimizers for Hitchin representations need not be unique. First we recall the Labourie map.

Fix a connected split real group $G$ of adjoint type with Cartan decomposition $\mathfrak{k}\oplus\mathfrak{p}$ and metric $\nu$ on $G/K$. Given a conformal structure $S$, let 
$$\mathbf{H}_G(S) = \bigoplus_{i=1}^l H^0(S,\mathcal{K}^{m_i})
$$ 
be the Hitchin base for $G$ with respect to a homogeneous basis $\{p_i\}$ of invariant polynomials on $\p$. Let $\mathbf{H}_G(\Sigma_g)$ be the bundle over $\mathbf{T}_g$ whose fiber over $S$ is $\mathbf{H}_G(S)$. Since $\nu$ defines a polynomial in $\pK$, $G$-Higgs bundles over $S$ with Higgs field $\phi$ satisfying $\nu(\phi^2)=0$ are taken by the Hitchin map onto a subspace $\mathbf{M}_G(S) \subset \mathbf{H}_G(S)$ of complex codimension $3g-3 = \mathrm{dim}(H^0(S, \mathcal{K}^2))$. Varying over $\mathbf{T}_g$, we obtain a codimension $3g-3$ subbundle $\mathbf{M}_G(\Sigma_g)\subset \mathbf{H}_G(\Sigma_g)$ that parametrizes the space of minimal maps associated to Hitchin representations. The Labourie map is the map
$$
L_G: \mathbf{M}_G(\Sigma_g) \to \mathrm{Hit}(\Sigma_g,G)
$$
to the Hitchin component of the representation variety, defined by applying the non-abelian Hodge correspondence to the Hitchin section for $S$ on the fiber $\mathbf{M}_G(S)$. Note that in section 1.1 we were using the simplified notation $\mathbf{M}_{n}(\cdot)=\mathbf{M}_{\textrm{PSL}(n,\R)}(\cdot)$, $\mathbf{H}_{n}(\cdot)=\mathbf{H}_{\textrm{PSL}(n,\R)}(\cdot)$, $L_n=L_{\textrm{PSL}(n,\R)}$. Labourie's existence result \cite[Theorem 1.0.3]{L1} shows that $L_G$ is surjective. Theorem A implies the following.

\begin{thm}
For every split real semisimple $G$ of rank at least $3$ with maximal compact subgroup $K$ and invariant metric $\nu$, the Labourie map $L_G$ is not injective.
\end{thm}
The theorem implies the following lemma.
\begin{lem}\label{lem: nosec}
There is no continuous section of the Labourie map.
\end{lem}
\begin{proof}
Suppose that $s:\textrm{Hit}(\Sigma_g,G)\to\textbf{M}_G(\Sigma_g)$ is a continuous section of $L_G$. The spaces $\textbf{M}_G(\Sigma_g)$ and $\mathrm{Hit}(\Sigma,G)$ are both topologically open balls of the same real dimension $(2g-2)\dim G$. In the case of $\mathrm{Hit}(\Sigma,G)$, this is because the Hitchin section from a fixed Hitchin base is a homoeomorphism onto $\mathrm{Hit}(\Sigma,G)$, and since $\mathrm{dim}\mathbf{T}_g = \mathrm{dim}H^0(\mathcal{K}^2)$, this also shows that the dimensions of $\textbf{M}_G(\Sigma_g)$ and $\textrm{Hit}(\Sigma_g,G)$ are the same. Since $s$ is a section, it is both injective and proper. By Brouwer's invariance of domain, $s$ is homeomorphism. But then the Labourie map is also a homeomorphism, which contradicts Theorem A.
\end{proof}

To prove Corollary A, we argue by contradiction, and suppose that for every Hitchin representation $\rho,$ there exists a unique equivariant minimizing minimal surface. If this is the case, we can define a section 
$$
s_\mathrm{min}:\textrm{Hit}(\Sigma_g,G)\to \mathbf{M}_G(\Sigma_g)
$$ 
that associates each representation to the conformal structure of the minimal surface and the point in the Hitchin base for the Higgs bundle data. By Lemma \ref{lem: nosec}, the proof of Corollary A will follow from the lemma below.
\begin{lem}\label{lem: continuous}
If every Hitchin representation had a unique equivariant minimal surface, the map $s_{\textrm{min}}$ would be continuous.
\end{lem}
The auxiliary result we use is contained in a paper of Tholozan \cite{Tho}.
\begin{defn}
Let $X$ and $Y$ be two metric spaces, and $(F_y)_{y \in Y}$ a family of maps $F_y:X\to \R$ depending continuously on $Y$ in the compact-open topology. We say that $(F_y)_{y \in Y}$ is locally uniformly proper if for every $y_0\in Y,$ there exists a neighbourhood $U$ of $y_0$ such that for any $C\in\R,$ there exists a compact set $K\subset X$ such that for all $y\in U$ and $x\in X\backslash K$, $F_y(x)>C.$
\end{defn}
\begin{lem}[Proposition 2.6 of \cite{Tho}]\label{lem: tho}
Let $X$ and $Y$ be two metric spaces, and $(F_y)_{y \in Y}$ a locally uniformly proper family of maps $F_y:X\to \R$ depending continuously on $Y$. Assume that each $F_y$ achieves its minimum at a unique point $x_m(y)\in X$. Then the map from $Y\to X$ defined by $$y\mapsto x_m(y)$$ is continuous.
\end{lem}
Fix a hyperbolic metric $g_0$ on $\Sigma_g$. Given a representation $\rho:\pi_1(\Sigma_g)\to G$ and $\gamma\in \pi_1(\Sigma_g)$, let $$\ell(\rho(\gamma))=\textrm{inf}_{x\in G/K} d(\rho(\gamma)x,x)$$ be the translation length for $\rho(\gamma)$ and $\ell_{g_0}(\gamma)$ the $g_0$-length of the geodesic representative of $\gamma$. If $\rho$ is a Hitchin representation, then it is well-displacing \cite[Theorem 1.01]{L1}, which in \cite{L1} means that there exist constants $A,B>0$ such that $$\ell(\rho(\gamma)) \geq A\ell_{g_0}(\gamma)-B.$$ More generally, it is proved by Guichard-Wienhard that Anosov representations, which include Hitchin representations, define quasi-isometric embeddings from the Cayley graph of $\pi_1(\Sigma_g)$ to $G$ (with a left-invariant metric), and are hence well-displacing \cite[Theorem 1.7]{GuW}. In the proof below, we will use Theorem 5.14 of \cite{GuW}, which says that if $\rho$ is Anosov, then there are constants $K,C>0$ and a neighbourhood $U$ of $\rho$ in $\textrm{Hom}(\pi_1(\Sigma_g),G)$ such that every $\rho'\in \textrm{Hom}(\pi_1(\Sigma_g),G)$ is a $(K,C)$-quasi-isometric embedding, from which it follows that the $\rho'$'s have uniform well-displacing constants $A,B>0.$

\begin{proof}[Proof of Lemma \ref{lem: continuous}]
We take $X=\mathbf{T}_g$, $Y=\textrm{Hit}(\Sigma_g,G),$ and $F_\rho=\mathbf{E}_\rho : \mathbf{T}_g\to \R.$ $\rho \mapsto \mathbf{E}_\rho$ is smooth in $\rho$; this is a consequence of the implicit function theorem \cite{Sle}. Assuming $(\mathbf{E}_\rho)_{\rho\in \textrm{Hit}(\Sigma_g,G)}$ is locally uniformly proper, Lemma \ref{lem: tho} asserts that the map $m:\textrm{Hit}(\Sigma_g,G)\to \mathbf{T}_g$ taking $\rho$ to the conformal structure of the area minimizing surface is continuous. 

Toward continuity of $s_{\textrm{min}}$, it suffices to work locally in $\textbf{M}_G(\Sigma_g)$, where we can write $\textrm{s}_{min}(\rho)=(S(\rho),\alpha(\rho)),$ with $\alpha(\rho)\in \mathbf{M}_G(S(\rho))$. Every Hitchin representation $\rho:\pi_1(\Sigma_g)\to G$ defines a section $\delta_\rho:\mathbf{T}_g\to \mathbf{H}_G(\Sigma_g)$ as follows. For each Riemann surface, we obtain a $G$-Higgs bundle from $\rho$ using $G$-NAH I, and then we apply the Hitchin map to get a point in $\mathbf{H}_G(\Sigma_g)$. Assembling these sections into a map $\delta:\textrm{Hit}(\Sigma_g,G)\times \mathbf{T}_g\to \mathbf{H}_G(\Sigma_g),$ $\delta(\rho,\cdot) = \delta_\rho(\cdot),$ it is well-understood that $\delta$ is continuous (essentially a consequence of the fact that harmonic maps vary smoothly \cite{Sle}). We express $$\alpha(\rho)=\delta(\rho,S(\rho))$$ and deduce that $s_{\textrm{min}}$ is continuous.

It remains to justify that $(\mathbf{E}_\rho)_{\rho\in \textrm{Hit}(\Sigma_g,G)}$ is locally uniformly proper, which amounts to going through Labourie's proof that each $\mathbf{E}_\rho$ is proper and observing that the estimates can be made locally uniform in $\rho$. This requires no new insight, so we only give a brief explanation and point to the relevant references. 

Let $S_0$ be a fixed marked Riemann surface structure on $\Sigma_g$. In his proof of \cite[Theorem 1.0.3]{L1}, Labourie shows $$\mathbf{E}_\rho(S) \geq A_\rho (\textrm{inter}(S,S_0))^2,$$ where $\textrm{inter}(\cdot,\cdot):\mathbf{T}_g^2\to \R$ is the intersection function (see \cite{Bon}), and $S\mapsto \textrm{inter}(S,S_0)$ is known to be proper on $\mathbf{T}_g$ (see \cite[Proposition 4]{Bon} or \cite[Proposition 6.2.4]{L1}). Stepping into the proof, $A_\rho$ depends only on the minimal well-displacing constant $A$ for $\rho$, which by \cite[Theorem 5.14]{GuW} is locally uniformly controlled with $\rho$. This establishes the result.
\end{proof}

\begin{proof}[Proof of Corollary A]
Assuming uniqueness of area minimizers, we construct the section $s_\mathrm{min}:\textrm{Hit}(\Sigma_g,G)\to \mathbf{M}_G(\Sigma_g)$.
We apply Lemma \ref{lem: nosec} and Lemma \ref{lem: continuous} to find a contradiction.

\end{proof}
\end{subsection}

\end{section}

\bibliographystyle{plain}
\bibliography{bibliography}

\end{document}